\documentclass[12pt]{article}
\usepackage{amsmath,amssymb,amsthm,tikz,tikz-cd}
\usepackage{mathabx}
\usepackage[mathscr]{euscript}
\usetikzlibrary{matrix,arrows}
\usepackage{adjustbox}
\usepackage[bookmarksopen,bookmarksdepth=4]{hyperref}
 \usepackage{relsize}
\usepackage[utf8]{inputenc}
\usepackage[english]{babel}
\usepackage[T1]{fontenc}
\usepackage{graphicx}
\usepackage{rotating}
\usepackage{faktor}
\usepackage[a4paper]{geometry}
\usepackage{xassoccnt}
\NewTotalDocumentCounter{totalfigures}
\NewTotalDocumentCounter{totaltables}

\DeclareAssociatedCounters{figure}{totalfigures}
\DeclareAssociatedCounters{table}{totaltables}

\hypersetup{
    colorlinks = true,
    linkcolor = {red},
}

\newcommand{\defi}[1]{\emph{\textbf{#1}}}

\newcommand{\Aa}{\mathbb{A}} 
\newcommand{\Nn}{\mathbb{N}}

\newcommand{\Rr}{\mathbb{R}}

\newcommand{\Tt}{\mathbb{T}}
\newcommand{\Ll}{\mathbb{L}}

\newcommand{\Id}{\mathrm{Id}}  
\newcommand{\Gg}{\mathbb{G}} 
 
\newcommand{\tx}[1]{\mathrm{#1}}

\def\G_#1{\mathfrak{#1}}

\newcommand{\icat}{$\infty$-category}
\newcommand{\icats}{$\infty$-categories}

\newcommand{\Map}{\tx{Map}}
\newcommand{\Hom}{\tx{Hom}}
\newcommand{\Sym}{\tx{Sym}}
\newcommand{\Mod}{\tx{Mod}}

\usepackage{tikz}
\usepackage{tikz-cd}
\usepackage{pgfplots}
\usetikzlibrary{calc}
\usetikzlibrary{shadows}
\usetikzlibrary{arrows}
\usetikzlibrary{patterns}

\usepackage{multicol}
\newtheorem{theor}{Theorem}[section]
\newtheorem{Th}[theor]{Theorem}
\newtheorem{Lem}[theor]{Lemma}
\newtheorem{Prop}[theor]{Proposition}

{  \theoremstyle{definition}
           \newtheorem{Def}[theor]{Definition}
           \newtheorem{RQ}[theor]{Remark}
           \newtheorem{Ex}[theor]{Example}
           
}

 \title{A Derived Lagrangian Fibration on the Derived Critical Locus}
 \author{Albin Grataloup \footnote{IMAG, Univ. Montpellier, CNRS, Montpellier, France \newline 
 albin.grataloup@umontpellier.fr}}

\begin{document}
\maketitle

\begin{abstract}
We study the symplectic geometry of derived intersections of Lagrangian morphisms. In particular, we show that for a functional $f : X \rightarrow \mathbb{A}_k^1$, the derived critical locus has a natural Lagrangian fibration $\textbf{Crit}(f) \rightarrow X$. In the case where $f$ is non-degenerate and the strict critical locus is smooth, we show that the Lagrangian fibration on the derived critical locus is determined by the Hessian quadratic form.   
\end{abstract}

\tableofcontents

\newpage

\section{Introduction}

In the context of derived algebraic geometry (\cite{LuDAG}, \cite{LuDAGV}, \cite{To1}, \cite{HAGI}, \cite{HAGII}), the notion of shifted symplectic structures was developed in \cite{PTVV} (see also \cite{Cal1} and \cite{Cal2}). This has proven to be very useful in order to obtain symplectic structures out of natural constructions. For example we obtain:
\begin{itemize}
\item shifted symplectic structures from transgression procedures (Theorem 2.5 in \cite{PTVV}), for example, the AKSZ construction. 
\item shifted symplectic structures from derived intersections of Lagrangians structures (Section 2.2 in \cite{PTVV}).
\item symplectic structures on various moduli spaces (Section 3.1 in \cite{PTVV}).
\item quasi-symplectic groupoids (see \cite{Xu03}) inducing shifted symplectic structures on the quotient stack as explained in \cite{Cal1}.
\item symmetric obstruction theory as defined in \cite{BF} from $(-1)$-shifted symplectic derived stacks (see \cite{STV} for the obstruction theory on derived stacks and \cite{PTVV} for the symmetric and symplectic enhancement thereof).
\item the $d$-critical loci as defined by Joyce in \cite{Joyce}. Every $(-1)$-shifted symplectic derived scheme induces a classical $d$-critical locus on its truncation (see Theorem 6.6 in \cite{Darboux}).   
\end{itemize} 

Another very useful construction in derived geometry is the derived intersection of derived schemes or derived stacks (see \cite{PTVV}). This includes many constructions such as:

\begin{itemize}
\item the derived critical locus of a functional (see \cite{PTVV} and \cite{Ve1}). For an action functional, this amounts to finding the space of solutions to the Euler-Lagrange equations, as well as remembering about the symmetries of the functional. 
\item $G$-equivariant intersections. This includes the example of symplectic reduction which can be expressed as the derived intersection of derived quotient stacks (see Section 2.1.2 in \cite{Cal4}).
\end{itemize}

In this paper, we make a more precise study of the shifted symplectic geometry of derived critical loci, and more generally of the derived intersections of Lagrangian morphisms. In particular, the main theorem (Theorem \ref{Th_DerivedIntersectionLagrangianFibration}) of this paper says  that whenever the Lagrangian morphisms $f_i : X_i \rightarrow Z$, $i=1..2$ look like "sections" in the sense that there exists a map $r: Z \rightarrow X$ such that the composition maps $r \circ f_i: X_i \rightarrow X$ are weak equivalences, then the natural morphism $X_1 \times_Z X_2 \rightarrow X$ is a Lagrangian fibration (see \cite{Cal2}). We then specialise this result to various examples and show in particular that, for the derived critical locus of a non-degenerate functional on a smooth algebraic variety, the non-degeneracy of the Lagrangian fibration is related to the non-degeneracy of the Hessian quadratic form of the functional. \\

This paper starts, in Subsection \ref{Sec_DerivedSymplecticGeometry},  by recalling the basic definitions and properties of shifted symplectic structures, Lagrangian structures and Lagrangian fibrations. We also recall, in Section \ref{Sec_RelativeCotangentComplex}, basic properties of the relative cotangent complexes of linear stacks that prove useful when we try to understand in more details the structure of Lagrangian fibrations on derived critical loci. 

In Section \ref{Sec_SymplecticGeometryoftheDerivedCriticalLocus} we start by recalling the fact that a derived intersection of Lagrangian structures in a $n$-shifted symplectic derived Artin stack is $(n-1)$-shifted symplectic. Then, in Subsection \ref{Sec_LagragianFibIntersection}, we state and prove the main theorem (Theorem \ref{Th_DerivedIntersectionLagrangianFibration}) that roughly says that if the Lagrangian morphisms look like sections (up to homotopy), then the natural projection from the derived intersection has a structure of a Lagrangian fibration. We then recall basic elements on the derived critical loci of a functional $f : X \rightarrow \Aa_k^1$, and then try to describe the Lagrangian fibration structure on the natural map $\textbf{Crit}(f) \rightarrow X$ obtained from the main theorem.

Section \ref{Sec_Examples} gives examples of applications of our main theorem. In particular, in Subsections \ref{Sec_1NonDegenerateCritPoint} and \ref{Sec_FamillyNonDegenerateCriticalLocus}, we give a better description of the Lagrangian fibration on the derived critical loci for non-degenerate functionals. We show that the non-degeneracy condition of the Lagrangian fibration of the derived critical locus of a non-degenerate functional on a smooth algebraic variety is given by the non-degeneracy of the Hessian quadratic form.        

 \paragraph{Acknowledgements:} I would like to thank Damien Calaque for suggesting this project; for all his help with it and for his revisions of this paper. I am also very grateful for everything he explained to me on the subject of derived algebraic geometry. I would also like to thank Pavel Safronov for his comments on the first version of this paper. This research has received funding from the European Research Council (ERC) under the European Union’s Horizon 2020 research and innovation programme (Grant Agreement No. 768679).
 
 \paragraph{Notation:}
 
 \begin{itemize}
 \item Throughout this paper $k$ denotes a field of characteristic $0$.
 \item $\textbf{cdga}$ (resp. $\textbf{cdga}_{\leq 0}$) denotes the \icat \ of commutative differential graded algebra over $k$ (resp. commutative differential graded algebra in non positive degrees). 
 \item $\textbf{cdga}^{gr}$ denotes the \icat \ of commutative monoids in the category of graded complexes $\textbf{dg}_k^{gr}$. 
 
 \item $A-\textbf{Mod}$ denotes the \icat \ of differential graded $A$-modules for $A \in \textbf{cdga}$.
 
 \item $\textbf{cdga}^{\epsilon-gr}$ denotes the \icat \ of graded mixed differential graded algebra. We denote the differential $\delta$ and the mixed differential $\epsilon$ or $d = d_{DR}$ in the case of the De Rham complex of a derived Artin stack $X$, denoted $\textbf{DR}(X)$. We refer to \cite{CPTVV} for the definitions of $\textbf{cdga}^{\epsilon-gr}$, $\textbf{cdga}^{gr}$, $\textbf{dg}_k^{gr}$ and the De Rham complex (see also \cite{PTVV} but with a different grading convention). 
 \item All the \icats \ above are localisations of model categories (see \cite{CPTVV} for details on these model structures and associated \icats ). All along, unless explicitly stated otherwise, every diagrams will be homotopy commutative, every functor will be $\infty$-functors and every (co)limits will be $\infty$-(co)limits. 
 \item For $X$ a derived Artin stack, $\textbf{QC}(X)$ denote the \icat \ of quasi-coherent sheaves on $X$. 

 \item In this paper derived Artin stack, denoted $\textbf{dSt}$, are defined as in \cite{HAGII}. In particular derived Artin stacks are locally of finite presentation over $\tx{Spec}(k)$.
 \item We denote by $\Ll_X$ the cotangent complex of a derived Artin stack $X$. We denote by $\Tt_X := \Ll_X^\vee := \tx{Hom}\left(\Ll_X, \_O_X \right)$ its dual.   
\end{itemize}

 \section{Derived Symplectic Geometry}
 \label{Sec_DerivedSymplecticGeometry}
 \subsection{Shifted Symplectic Structures}

Before going to symplectic structures, we make a short recall of differential calculus and (closed) differential $p$-forms in the derived setting. Recall from \cite{PTVV} that there are classifying stacks $\_A^p (\bullet, n)$ and $\_A^{p,cl}(\bullet, n)$ of respectively the space of $n$-shifted differential $p$-forms and the space of $n$-shifted closed differential $p$-forms. We use the grading conventions used in \cite{CPTVV}. On a derived affine scheme \textbf{Spec}(A), the space of $p$-forms of degree $n$ and the space of closed $p$-forms of degree $n$ are defined respectively by 
\[\_A^p(A, n) := \Map_{\textbf{cdga}^{gr}} \left( k[-n-p](-p), \textbf{DR}(A) \right)\]

 and 
 
 \[ \_A^{p,cl}(A, n) := \Map_{\textbf{cdga}^{\epsilon-gr}} \left( k[-n-p](-p), \textbf{DR}(A) \right).\] 
 
 From \cite{CPTVV}, the de Rham complex of $A$, denoted $\textbf{DR}(A)$, can be described, as a graded complex, by $\textbf{DR}(A)^\# \simeq \Sym_A \Ll_A[-1](-1)$ where $(-)^\# : \textbf{cdga}^{\epsilon-gr} \rightarrow \textbf{cdga}^{gr}$ is the functor forgetting the mixed structure (we refer to \cite{CPTVV} for more details on the de Rham complex).   \\

All along, we denote the internal differential, i.e. the differential on $\Ll_A$, by $\delta$ and the mixed differential, i.e. the de Rham differential, by $d$. \\

By definition, the space of $p$-forms of degree $n$ on a derived stack $X$ is the mapping space $\Map_{\textbf{dSt}}\left( X, \_A^{p}(\bullet, n) \right)$ and the space of closed $p$-forms of degree $n$ on $X$ is $\Map_{\textbf{dSt}}\left( X, \_A^{p,cl}(\bullet, n) \right)$. Now the following proposition says that in the case where $X$ is a derived Artin  stack, the spaces of shifted differential forms are spaces of sections of quasi-coherent sheaves on $X$.\\

\begin{Prop}[Proposition 1.14 in \cite{PTVV}]\
\label{Prop_1.14PTVV}
Let $X$ be a derived Artin stack over k and $\Ll_X$ be its cotangent complex over k. Then there is an equivalence

$$ \_A^{p}(X,n) \simeq \Map_{\textbf{QC}(X)} \left( \_O_X, \Lambda^p \Ll_X [n]\right).$$

\end{Prop}

\begin{RQ} More concretely, we have from \cite{Cal1} and \cite{CPTVV} an explicit description of (closed) $p$-forms of degree $n$ on a geometric derived stack $X$. A $p$-form of degree $n$ is given by a global section $\omega \in \textbf{DR}(X)_{(p)} [n+p]  \simeq \Rr \Gamma \left(\left( \bigwedge^p \Ll_X \right) [n] \right) $ such that $\delta \omega = 0$. A closed $p$-form of degree $n$ is given by a semi-infinite sequence $\omega = \omega_0 + \omega_1 + \cdots$ with $\omega_i  \in \textbf{DR}_{(p+i)} [n+p]= \Rr \Gamma \left( \left( \bigwedge^{p+i}  \Ll_X \right) [n-i]\right)$ such that $\delta \omega_0 = 0$ and $d \omega_{i} = \delta \omega_{i+1}$. 

Equivalently, being closed means that $\omega$ is closed for the total differential $D = \delta + d$ in the bi-complex $\textbf{DR}(X)_{\geq p}[n] \simeq \Rr \Gamma \left( \prod_{i \geq 0} \left( \bigwedge^{p+i} \Ll_X \right) [n] \right)$ whose total degree is given by $n+p+i$. Note that the conditions imposed on $\omega$ are equivalent to saying that $\omega$ is a cocycle of degree $n+p$ for the total differential. \\

In general, we can also describe the spaces of (closed) differential forms as $\_A^{p}(X, n) \simeq \left| \textbf{DR}_{(p)}(X)[n+p] \right|$ 
and $\_A^{p,cl}(X, n) \simeq \left| \prod_{i\geq p} \textbf{DR}_{(p+i)}(X)[n+p] \right| $, where $\prod_{i\geq 0} \textbf{DR}_{(p+i)}(X)[n]$ is endowed with the total differential.
\end{RQ}

\begin{RQ}
\label{RQ_RelativeShiftedForms}
Given a map of derived Artin stack $f : Y \rightarrow X$, we define $\_A^{p,(cl)}(Y/X, n)$, the space of $n$-shifted (closed) $p$-forms on $Y$ relative to $X$, to be the homotopy cofiber of the natural map $f^* : \_A^{p,(cl)}(X,n) \rightarrow \_A^{p,(cl)}(Y,n)$. For instance $n$-shifted relative $p$-forms are equivalent to the derived global sections of $\left( \bigwedge^p \Ll_{\faktor{Y}{X}} \right)[n]$, with the relative cotangent complex $\Ll_{\faktor{Y}{X}}$ defined as the homotopy cofiber of the natural map $f^*\Ll_X \rightarrow \Ll_Y$. We refer to \cite{CPTVV} for more details on the relative $n$-shifted (closed) $p$-forms and the relative version of the De Rham complex. 
\end{RQ}

We say that a $p$-form, $\omega_0$, of degree $n$ can be lifted to a closed $p$-form of degree $n$ if there exists a family of $(p+i)$-forms $\omega_i$ of degree $n-i$ for all $i > 0$, such that $\omega = \omega_0 + \omega_1 + \cdots $ is closed in $DR(X)_{\geq p}[n]$ (i.e. $D \omega = 0$). In this situation, we can see that $d \omega_0$ is in general not equal to $0$ but is homotopic to $0$ ($d\omega_0 =  D\left(- \sum_{i>0} \omega_{p+i} \right)$). The choice of such a homotopy is the same as a choice of a closure of the $p$-form of degree $n$. Being closed is therefore no longer a property of the underlying $p$-form of degree $n$ but a structure given by a homotopy between $d\omega_0$ and zero. The collection of all closures of a $p$-form of degree $n$ forms a space:

\begin{Def}
\label{Def_SpaceofKey}
Let $\alpha \in \_A^{p}(X,n)$ then the space of all closures of $\alpha$ is called the \defi{space of keys} of $\alpha$ denoted $\textbf{key}(\alpha)$. It is given by the homotopy pull-back:

\begin{equation}
\begin{tikzcd}
\textbf{key}(\alpha) \arrow[r] \arrow[d] & \_A^{p,cl}(X,n) \arrow[d]\\
\star \arrow[r, "\alpha"] & \_A^p(X,n)
\end{tikzcd}
\end{equation}
\end{Def}

\vspace{0.3cm}

The mixed differential of the de Rham graded mixed complex induces a map: 
\[d : \_A^p (X,n) \rightarrow \_A^{p+1, cl}(X,n)\] 

We are now turning toward symplectic geometry. We now know what are (shifted) closed 2-forms we only need to mimic the notion of non-degeneracy to define symplectic structures. 

\begin{Def}[Non-Degenerate 2-Form of Degree $n$] For a derived Artin $n$-stack $X$, the cotangent complex $\Ll_X$ is dualisable. Therefore there is a tangent complex $\Tt_X = \Ll_X^\vee$. We say that a (closed) 2-form of degree $n$ is \defi{non-degenerate} if the (underlying) 2-form $\omega_0$ of degree $n$ induces a quasi-isomorphism:
$$\omega_0^\flat : \Tt_X \rightarrow \Ll_X[n]$$  

We denote by $\_A^{2, nd}(X,n) $ the subspace of $\_A^2(X,n)$ generated by the non-degenerate $n$-shifted $2$-forms. 
\end{Def}

\begin{Def}[Shifted Symplectic Forms]
\label{Def_SymplecticForms} A \defi{$n$-shifted symplectic structure} is a non-degenerate $n$-shifted closed 2-form on $X$. $n$-shifted symplectic structures form a space defined as the pullback: 
\[ \textbf{Symp}(X,n) := \_A^{2, nd}(X,n) \times_{\_A^2(X,n)} \_A^{2,cl}(X,n) \]

\end{Def}

The standard example of symplectic manifold is the cotangent bundle. In our setting, we can speak of $n$-shifted cotangent stacks. It is a derived stack defined as linear stack associated to $\Ll_X[n]$, $T^*[n]X := \Aa( \Ll_X[n] )$ (see Definition \ref{Def_LinearStacks}). It comes with a natural morphism $\pi_X : T^*[n]X \rightarrow X$. We refer to \cite{Cal2} for a general account of shifted symplectic geometry on the cotangent stack.   

\begin{Def}[Linear Stacks]
\label{Def_LinearStacks}
Given $\_F \in \textbf{QC}(X)$ a quasi-coherent sheaf over a derived Artin stack, we can construct a \defi{linear stack} denoted $\Aa (\_F)$, defined as a derived stack over $X$ by:
\[ \Aa(\_F) \left( f: \textbf{Spec}(A) \rightarrow X \right) := \textbf{Map}_{\textbf{A-\textbf{Mod}}} \left( A, f^* \_F \right)\]
\end{Def}
 
\begin{RQ}
\label{RQ_LinearStackAndSection}
A morphism $Y \rightarrow T^*[n]X$ is determined by the induced morphism $f : Y \rightarrow X$ (by composition with $\pi_X$) and a section $s : Y \rightarrow f^* T^*[n]X$ which corresponds to an element  $s \in \Map_{\textbf{QC}(Y)} \left( \_O_Y, f^* \Ll_X[n]\right)$ (see \cite{Cal2} Section 2 for more details). In the case of a section $s: X \rightarrow T^*[n]X$, we get the identity $\tx{Id}: X \rightarrow X$ and a section $s_1 \in \Map_{\textbf{QC}(X)} \left( \_O_X,  \Ll_X[n]\right) \simeq \_A^1 (X,n)$. This shows, using Proposition \ref{Prop_1.14PTVV}, that the space of sections of $T^*[n]X$ is exactly the space of $1$-forms of degree $n$ as expected.  
\end{RQ}

\begin{Ex}
\label{Ex_CotangentComplexisSymplectic}
As in the classical case, we can construct the canonical Liouville 1-form. Consider the identity $\Id : T^*[n]X \rightarrow T^*[n] X$. It is determined by the projection $\pi : T^*[n]X \rightarrow X$ and a section $\lambda_X \in \Map_{\textbf{QC}(T^*[n]X)} \left( \_O_{T^*[n]X}, \pi^* \Ll_{X}[n]\right)$. Since we have a natural map $\pi^* \Ll_{X}[n] \rightarrow \Ll_{T^*[n]X}[n]$, $\lambda_X$ induces a 1-form on $T^*[n]X$ called the tautological $1$-form. This $1$-form induces a closed $2$-form $d \lambda_X$ which happens to be non-degenerate (see \cite{Cal2} Section 2.2 for a proof of the non-degeneracy). 
\end{Ex}

This symplectic structure on the cotangent is universal in the sense that it satisfies the usual universal property.

\begin{Lem}
\label{Lem_UniversalPropTautological1Form}
Given a 1-form $\alpha : X \rightarrow T^*[n]X$, we have that $\alpha^* \lambda_X = \alpha$. 
\end{Lem} 
 
 \begin{proof} 
 
 In general, if we take $f: X \rightarrow Y$, the pull-back of a $n$-shifted $1$-form, $\beta$, is described by: 
 
 $$\begin{tikzcd}
 T^*[n]X \arrow[dr, shift left] & f^* T^*[n] Y \arrow[l,"(df)^*"'] \arrow[d, shift left] \arrow[r] & T^*[n] Y \arrow[d, shift left] \\
  & X \arrow[ul, dashed, shift left, "f^* \beta"] \arrow[u, shift left, dashed] \arrow[r, "f"] & Y \arrow[u, dashed, shift left, " \beta"] 
 \end{tikzcd}
 $$
 
 Taking into account the fact that $\lambda$ factors through $\pi_X^* T^*[n]X$, we consider the following diagram:

$$\begin{tikzcd}
 T^*[n] X & \alpha^* T^* T^*[n] X \arrow[l, "(d\alpha)^*"'] \arrow[r] & T^* T^*[n] X   \\
  & T^*[n] X = \alpha^* \pi_X^* T^*[n] X \arrow[ul, dashed, "\tx{Id}"] \arrow[r] \arrow[u, dashed, "(d\pi_X)^*"'] \arrow[d, shift left] & \pi_X^* T^*[n] X  \arrow[d, shift left] \arrow[u, "(d\pi_X)^*"] \\
 &X \arrow[u, dashed, shift left, "\tilde{\lambda}"] \arrow[r, "\alpha"]  & T^*[n] X \arrow[u, dashed, shift left, "\lambda"]
\end{tikzcd} 
 $$
 
 This proves that the pull-back along $\alpha$ of $\lambda_X$ seen as a 1-form of degree $n$ on $T^*[n]X$ is the same as the pull-back along $\alpha$ of the section $\lambda_X : T^*[n]X \rightarrow \pi_X^* T^*[n]X$.\\
 
 We denote by $\alpha_1$ the associated section in $\Map_{\textbf{QC}(X)} \left( \_O_X, \Ll_X[n] \right)$ of degree $n$. There is a one-to-one correspondence between sections of $\pi_X : T^*[n]X \rightarrow X$ and points of $\Map_{\textbf{QC}(X)} \left( \_O_X, \Ll_X[n] \right)$.
  Now we use the fact that $\tx{Id}\circ \alpha = \alpha$: \begin{itemize}
  \item On the one hand, $\alpha$ is completely described by $\alpha_1 \in \Map_{\textbf{QC}(X)} \left( \_O_X, \Ll_X[n] \right)$. 
  
  \item On the other hand, the map $\tx{Id}: T^*[n] X \rightarrow T^*[n]X$ is described by the projection $\pi : T^*[n]X \rightarrow X$ and the section $\lambda_X \in \Map_{\textbf{QC}(T^*[n]X)} \left( \_O_{T^*[n]X}, \pi_X^* \Ll_X \right)$. Therefore the composition $\tx{Id} \circ \alpha$ is also a section of $\pi_X$ and is described by $ \alpha^* \lambda_X \in \Map_{\textbf{QC}(X)} \left( \_O_X, \Ll_X[n] \right)$.
  \end{itemize}

 This proves that $\alpha^* \lambda_X = \alpha_1$. Since these maps characterise the sections of $\pi_X$ they represent, we have $\alpha^* \lambda_X = \alpha$.  
 \end{proof}

 \subsection{Lagrangian Structures}  
 \label{Sec_LagrangianStructures} 
We recall from  \cite{PTVV} the definition and standard properties of Lagrangian structures. We also provide proves of some results which are well known to the expert but are not written as far as we know.  

\begin{Def}[Isotropic Structures]
\label{Def_DerivedIsotropicStrucuture}
Let $ f : L \rightarrow X$ be a map of derived Artin stacks. An \defi{isotropic structure on $f$} is a homotopy, in $\_A^{2,cl}(L,n)$, between $f^* \omega$ and $0$ for some $n$-shifted symplectic structure $\omega : \star \rightarrow \textbf{Symp}(X,n)$. Isotropic structures on $f$ form a space described by the homotopy pull-back:

$$ \begin{tikzcd}
\textbf{Iso}(f,n) \arrow[r] \arrow[d] & \textbf{Symp}(X,n) \arrow[d, "f^*"] \\
\star \arrow[r, "0"] & \_A^{2,cl}(L,n)
\end{tikzcd}$$

If we fix a given $n$-shifted symplectic structure $\omega:  \star \rightarrow \textbf{Symp}(X,n)$, we can define the space of isotropic structures on $f$ at $\omega$ defined by:

\[\textbf{Iso}(f, \omega) := \textbf{Iso}(f,n)\times_{\textbf{Symp}(X,n)}\star \simeq \star \times_{0, \ \_A^{2,cl}(X,n), \ f^*\omega} \star  \]
\end{Def}

\begin{RQ}
\label{RQ_DescriptionIsotropicStructures}
More explicitly, an isotropic structure is given by a family of forms of total degree $(p+n-1)$, $(\gamma_i)_{i \in \Nn}$ with $ \gamma_i \in \textbf{DR}(L)_{(p+i)}[p+n+i-1]$, such that $\delta \gamma_0 = f^*\omega_0$ and $\delta \gamma_i + d \gamma_{i-1} = f^*\omega_i$. This can be rephrased as $D \gamma = f^* \omega$, thus $\gamma$ is indeed a homotopy between $f^*\omega$ and 0.     
\end{RQ}

\begin{Def}[Lagrangian Structures]

\label{Def_DerivedLagrangianStructure}
An isotropic structure $\gamma$ on $f: L \rightarrow X$ is called a \defi{Lagrangian structure on f} if the leading term, $\gamma_0$, viewed as an isotropic structure on the morphism $\Tt_L \rightarrow f^* \Tt_X$, is non-degenerate. We say that $\gamma_0$ is \defi{non-degenerate} if the following null-homotopic sequence (homotopic to 0 via $\gamma_0$) is fibered:

\begin{equation}
\label{Dia_NonDegeneracyLagrangian}
\begin{tikzcd}
\Tt_L \arrow[r] \arrow[rr, bend right, "(f^*\omega_0)^{\flat}"'] & f^*\Tt_X \simeq f^* \Ll_X[n] \arrow[r] &\Ll_L[n]
\end{tikzcd}
\end{equation}

The space of $n$-shifted Lagrangian structures on $f$ is denoted $\textbf{Lag}(f,n)$. There are natural morphisms of spaces $\textbf{Lag}(f,n) \rightarrow \textbf{Iso}(f,n) \rightarrow \textbf{Symp}(X,n)$.
\end{Def}

\begin{RQ}
\label{RQ_NonDegeneracyLagrangian}

To say that that sequence is fibered can be reinterpreted as a more classical condition involving the conormal. Since $\textbf{QC}(X)$ is a stable \icats , the homotopy fiber of $f^* \Ll_X [n] \rightarrow \Ll_L[n]$ is denoted $\Ll_{\faktor{L}{X}}[n-1]:=\Ll_f[n-1]$ and the non-degeneracy condition can be rephrased by saying that the natural map $\Theta_f : \Tt_L \rightarrow \Ll_f [n-1]$ is a quasi-isomorphism.   
\end{RQ}

\begin{RQ}
	To simplify the notations, we will abusively say that a morphism $f: X \rightarrow Y$ \underline{\textbf{is}} Lagrangian when we consider $f$ together with a fixed Lagrangian structure on $f$ and a fixed symplectic structure $\omega$.    
\end{RQ}
\begin{Ex}
\label{Ex_LagrangianSectionof1form}
A $1$-form of degree $n$ on an Artin stack $X$ is equivalent to a section $\alpha : X \rightarrow T^*[n]X$. This section is a Lagrangian morphism if and only if $\alpha$ admits a closure, i.e. $\textbf{Key}(\alpha)$ is non-empty. This is Theorem 2.15 in \cite{Cal2}.  
\end{Ex} 

\begin{Prop}
\label{Prop_SpaceofKeyequivalenttoSpaceofIsotropicStructures}
There is a canonical homotopy equivalence $\textbf{Iso}(\alpha) \rightarrow \textbf{Key}(\alpha)$ between the space of isotropic structures on the $1$-form $\alpha$ and the space of keys of $\alpha$.
\end{Prop}

\begin{proof}
\begin{equation}
\begin{tikzcd}
\textbf{key}(\alpha) \arrow[r] \arrow[d] & \_A^{1,cl}(X,n) \arrow[r] \arrow[d] & \star \arrow[d, "0"]\\
\star \arrow[r,"\alpha"]& \_A^1(X,n) \arrow[r, "d_{dR}"] & \_A^{2,cl}(X,n)
\end{tikzcd}
\end{equation}

The leftmost square is Cartesian by definition of $\textbf{key}(\alpha)$ in Definition \ref{Def_SpaceofKey}. By definition, the pull-back of the outer square is $\textbf{Iso}(\alpha)$ because $d_{dR}\alpha = \alpha^* \omega$ (by universal property of the Liouville 1-form, Lemma \ref{Lem_UniversalPropTautological1Form}). 
It turns out that the rightmost square is also Cartesian. This is simply saying that the space of closed $1$-forms of degree $n$ is the same as the space of $1$-forms of degree $n$ whose de Rham differential is homotopic to $0$. We obtain that $\textbf{key}(\alpha)$ and $\textbf{Iso}(\alpha)$ are both pull-backs of the outer square and therefore are canonically homotopy equivalent. 
\end{proof}

 \begin{RQ}
 \label{RQ_Closed1FormAreNonDegenerateIsotropicStructure}
It turns out that Theorem 2.15 in \cite{Cal2} says that all the isotropic structures on $\alpha$ (or equivalently the lifts of $\alpha$ to a closed form) are in fact non-degenerate, which implies the statement in Example \ref{Ex_LagrangianSectionof1form} and even that the space of Lagrangian structures on $\alpha$ is equivalent to the space of keys of $\alpha$. 
 \end{RQ}

 \begin{Lem}[Example 1.26 in \cite{Cal1}]
 \label{Lem_LagrangianOveraPoint}
 Consider the map $X \rightarrow \star_n$ where $\star_n$ is the point endowed with the canonical $n$-shifted symplectic structure given by 0. Then a Lagrangian structure on this map is equivalent to an $(n-1)$-shifted symplectic structure on $X$. \\
 \end{Lem}
 \begin{proof}

Pick an isotropic structure $\gamma$ on $p$. We know that $\gamma$ is a homotopy between 0 and 0 which means that $D\gamma = 0$. Therefore $\gamma$ is a closed 2-form of degree $n-1$. We want to show that $\gamma$ is non-degenerate as an isotropic structure if and only if it is non-degenerate as a closed 2-form on $X$. The non-degeneracy of the Lagrangian structure, as described in Remark \ref{RQ_NonDegeneracyLagrangian}, corresponds to the requirement that the natural map $ \Tt_X \rightarrow \Ll_X[n-1]$ is a quasi-isomorphism. This map depends on $\gamma_0$ and we want to show that this map is in fact $\gamma_0^{\flat}$. This map is the natural map that fits in the following homotopy commutative diagram:

$$ \begin{tikzcd}[row sep = small]
\Tt_X \arrow[r]\arrow[d] & \Ll_X[n-1] \arrow[r] \arrow[d] & 0 \arrow[d] \\
0 \arrow[r, "0^{\flat}"] & 0 \arrow[r] & \Ll_X [n] 
\end{tikzcd}$$
  We can show that by strictifying the homotopy commutative diagram:

$$\begin{tikzcd}[row sep = small]
\Tt_X \arrow[dr] \arrow[drr, "p^*\omega^{\flat}=0"] & &  \\
 & 0 \arrow[r]& \Ll_X[n]
\end{tikzcd}
$$
Note that this diagram is already commutative but we see it as homotopy commutative using the homotopy $\gamma_0$. We use the homotopy $\gamma_0$ to strictify the previous diagram and we obtain:

$$\begin{tikzcd}[row sep = small]
\Tt_X \arrow[dr, "\gamma_0^{\flat} + 0"'] \arrow[drr, "p^*\omega^{\flat}=0"] & &  \\
 & \Ll_X[n-1]\oplus \Ll_X[n] \arrow[r, "\tx{pr}"']& \Ll_X[n]
\end{tikzcd}
$$

The homotopy fiber and also strict fiber of the projection $\tx{pr} : \Ll_X[n-1] \oplus \Ll_X[n] \rightarrow \Ll_X[n]$ is $\Ll_X [n-1]$, and therefore the natural map we obtain is $\gamma_0^{\flat} : \Tt_X \rightarrow \Ll_X[n-1]$.\\

Since the non-degeneracy condition of the isotropic structure $\gamma$ is the same as saying that the map $\gamma_0^{\flat}$ is a quasi-isomorphism, we have shown that an isotropic structure $\gamma$ is an $(n-1)$-shifted symplectic structure on $X$ if and only if it is non-degenerate as an isotropic structure on $X \rightarrow \star_n $.
 \end{proof}

 \begin{Def}[Lagrangian Correspondence, \cite{Cal4}]
 \label{Def_LagrangianCorrespondence}
 
 Let $X$ and $Y$ be derived Artin stacks with $n$-shifted symplectic structures. A \defi{Lagrangian correspondence} from $X$ to $Y$ is given by a derived Artin stack $L$ with morphims 
 
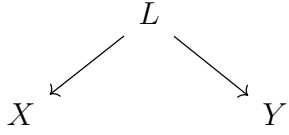
 $$ \begin{tikzcd} 
 &L \arrow[dl] \arrow[dr]& \\
   X & & Y
   \end{tikzcd}$$
   and a Lagrangian structure on the map $L \rightarrow X \times \widebar{Y}$ where $X \times \widebar{Y}$ is endowed with the $n$-shifted symplectic structure $\pi_X^* \omega_X - \pi_Y^* \omega_Y$. For example, a Lagrangian structure on $L \rightarrow X$ is equivalent to a Lagrangian correspondence from $X$ to $\star$.
 \end{Def}

 As explained in \cite{Cal4} Section 4.2.2, these Lagrangian correspondences can be composed. If we take $X_0$, $X_1$ and $X_2$ derived Artin stacks with symplectic structures and $L_{01}$ and $L_{12}$ Lagrangian correspondences from respectively $X_0$ to $X_1$ and $X_1$ to $X_2$. We can produce a Lagrangian correspondence $L_{02}$ from $X_0$ to $X_2$ by setting $L_{02} := L_{01} \times_{X_1} L_{12}$. 
 
 $$
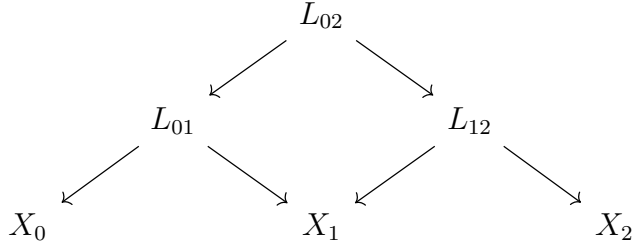
\begin{tikzcd}
  & & L_{02} \arrow[dr] \arrow[dl] & & \\
  &L_{01} \arrow[dr] \arrow[dl] & & L_{12} \arrow[dr] \arrow[dl]&  \\
  X_0& &X_1 & &X_2
 \end{tikzcd}$$

 \subsection{Lagrangian Fibration}\label{sec:lagrangian-fibration}\

We recall in this section the definition and standard properties of Lagrangian fibrations (\cite{Cal1} and \cite{Cal2}).  

\begin{Def}
\label{def: lagrangian fibration}
Let $f : Y \rightarrow X$ be a map of derived Artin stacks and $\omega$ a symplectic structure on $Y$. A \defi{Lagrangian fibration} on $f$ is given by: 

\begin{itemize}
\item A homotopy, denoted $\gamma$, between ${\omega}_{/X}$ and $0$, where ${\omega}_{/X}$ is the image of $\omega$ under the natural map $\_A^{2,cl}(Y,n) \rightarrow \_A^{2,cl}(Y/X, n)$ (see Remark \ref{RQ_RelativeShiftedForms}) for some $n$-shifted symplectic structure $\omega : \star \rightarrow \textbf{Symp}(Y,n)$. This forms a space of isotropic fibrations: 

\[ \begin{tikzcd}
\textbf{IsoFib}(f,n)  \arrow[r]  \arrow[d] & \textbf{Symp}(Y,n) \arrow[d] \\
\star \arrow[r, "0"] & \_A^{2,cl} ( Y/X, n) 
\end{tikzcd}\]

\item A non-degeneracy condition which says that the following sequence (homotopic to 0 via $\gamma_0$) is fibered: 

\begin{equation*}
\label{Eq_LagrangianFiberSequence}
\Tt_{Y/X} \rightarrow \Tt_Y \simeq \Ll_Y [n] \rightarrow \Ll_{Y/X}[n]  
\end{equation*}
\end{itemize}

In particular, the non-degeneracy condition can be rephrased by saying that there is a canonical quasi-isomorphism $\alpha_f : \Tt_{Y/X} \rightarrow f^*\Ll_{X}[n]$ (similar to the criteria for Lagrangian morphism in Remark \ref{RQ_NonDegeneracyLagrangian}) that makes the following diagram commute: 

\begin{equation}
\label{Dia_LagrangianFibrationNonDegeneracy}
\begin{tikzcd}
\Tt_{\faktor{Y}{X}} \arrow[d] \arrow[r, "\alpha_f"] & f^* \Ll_X[n] \arrow[d] \arrow[r] & 0 \arrow[d] \\
\Tt_Y \arrow[r, "\omega^{\flat}"] & \Ll_Y[n] \arrow[r] & \Ll_{\faktor{Y}{X}}[n] 
\end{tikzcd}
\end{equation}

The subspace of $\textbf{IsoFib}(f,n)$ generated by the non-degenerate objects is the space of Lagrangian fibration structures on $f$ and is denoted by $\textbf{LagFib}(f,n)$. There are natural maps $\textbf{LagFib}(f,n) \rightarrow \textbf{IsoFib}(f,n) \rightarrow \textbf{Symp}(Y,n)$.\\

Similarly to the Lagrangian case, we can fix a $n$-shifted symplectic structure on $Y$ and define Lagrangian and isotropic fibration of $f$ at a given $\omega$: 

\[ \textbf{IsoFib}(f, \omega) = \textbf{IsoFib}(f,n) \times_{\textbf{Symp}(Y,n)} \star  \simeq \star \times_{0,\_A^{2,cl}(Y/X,n), \omega_{/X}} \star  \]  
\end{Def} 

\begin{RQ}
	To simplify the notations, we will abusively say that a morphism $f: X \rightarrow Y$ \underline{\textbf{is}} a Lagrangian fibration when we consider $f$ together with a fixed shifted symplectic structure $\omega$ and a fixed structure of Lagrangian fibration on $f$ at $\omega$.    
\end{RQ}

\begin{Ex}
\label{Ex_CotangentLagrangianFibration}
The natural projection $\pi_X : T^*[n]X \rightarrow X$ is a Lagrangian fibration.  The Liouville 1-form is a section of $\pi_X^* \Ll_X[n]$ which is part of the fiber sequence: 

$$ \pi_X^* \Ll_X[n] \rightarrow \Ll_{T^*[n]X} [n] \rightarrow \Ll_{\faktor{T^*[n]X}{X}}[n]$$

Thus the 1-form induced by $\lambda_X $ in $\Ll_{\faktor{T^*[n]X}{X}}[n]$ is homotopic to 0. The non-degeneracy condition is more difficult and is proven in Section 2.2.2 of \cite{Cal2}. It turns out that the morphism expressing the non-degeneracy condition, $\alpha_{\pi_X}$, is given by a canonical construction (Proposition \ref{Prop_RelativeCotangentComplexForLinearStacks}) which does not depend on the symplectic structure. This is the content of Proposition \ref{Prop_CotangentNonDegeneracyFibrationIsNatural}.  \\
 
\end{Ex}
\begin{Lem}
\label{Lem_LagragianFibrationFromAPoint}
Let $x: \star_n \rightarrow X$ be a point of $X$. Then, given a Lagrangian fibration structure on $x$, the non-degeneracy condition is given by a quasi-isomorphism $x^* \Tt_X \rightarrow x^*\Ll_X[n+1]$. 
\end{Lem}
\begin{proof}
The Lagrangian fibration structure on $\star_n \rightarrow X$ is a homotopy between 0 and itself in $\_A^{2,cl}(\faktor{\star_n}{X},n)$. As in the proof of Lemma \ref{Lem_LagrangianOveraPoint}, this is given by an element $\gamma \in \_A^{2,cl}(\faktor{\star_n}{X},n-1)$. Similarly to what was done in the proof of Lemma \ref{Lem_LagrangianOveraPoint}, we can show that $\gamma$ is non-degenerate as a Lagrangian fibration if and only if it is non-degenerate as a closed 2-form of degree $n$. Again it boils down to the fact that the natural morphism in the non-degeneracy criteria for Lagrangian fibrations is in fact $\gamma_0^{\flat} : \Tt_{\faktor{\star_n}{X}} \rightarrow \Ll_{\faktor{\star_n}{X}}[n-1]$.  

Moreover, we have natural equivalences, $\Tt_{\faktor{\star_n}{X}} \simeq x^*\Tt_X [-1]$ and $ \Ll_{\faktor{\star_n}{X}}[n-1] \simeq x^*\Ll_X[n]$ because the sequence
$$ \begin{tikzcd}
\Tt_{\faktor{\star_n}{X}} \arrow[r] & \Tt_{\star_n} \simeq 0 \arrow[r] &  x^* \Tt_X[n]
\end{tikzcd} $$

is fibered. This concludes the proof. 
\end{proof}

\subsection{Relative Cotangent Complexes of Linear Stacks}
\label{Sec_RelativeCotangentComplex}
 This section is devoted to the study of relative cotangent complexes of linear stacks. Given $\_F \in \textbf{QC}(X)$ a dualisable quasi-coherent sheaf over a derived Artin stack $X$, we consider its associated linear stack, $\Aa (\_F)$ (see Definition \ref{Def_LinearStacks}) and the goal of this section is to describe $\Ll_{\faktor{\Aa (\_F)}{X}}$ and its functoriality in $\_F$ and $X$.

\begin{Prop} 
\label{Prop_RelativeCotangentComplexForLinearStacks}
 Let $X$ be a derived Artin stack and $\_F \in \textbf{QC}(X)$ a dualisable quasi-coherent sheaf on $X$. We denote $\pi_X : \Aa (\_F) \rightarrow X$ the natural projection. Then we have:
 
 \[ \Ll_{\pi_X} \simeq \Ll_{\faktor{\Aa(\_F)}{X}} \simeq \pi_X^* \_F^\vee \]  
 \end{Prop}
\begin{proof}
We will show the result for any $B$-point $y : \textbf{Spec}(B) \rightarrow \Aa (\_F )$ and we write $x = \pi \circ y : \textbf{Spec}(B) \rightarrow X$. We will show that for all $M \in B-\Mod$ connective, we have 

\[ \Hom_{B-\Mod} \left( y^*\Ll_{\faktor{\Aa (\_F)}{X} }, M \right) \simeq \Hom_{B-\Mod} \left(x^* \_F^\vee, M \right) \] 

First we observe that $\Hom_{B-\Mod} \left( y^*\Ll_{\faktor{\Aa (\_F)}{X} }, M \right) $ is equivalent, using the universal property of the cotangent complex, to the following homotopy fiber at $y$:

\[  \tx{hofiber}_{y} \left( \Hom_{\faktor{\textbf{dSt}}{X}} \left( \textbf{Spec}(B \oplus M), \Aa (\_F ) \right)\rightarrow \Hom_{\faktor{\textbf{dSt}}{X}} \left( \textbf{Spec}(B), \Aa (\_F ) \right) \right) \]

with $B \oplus M$ denoting the square zero extension and $\textbf{Spec}(B \oplus M) \rightarrow X$ being the composition: 

\[ \begin{tikzcd}
\textbf{Spec}(B \oplus M) \arrow[r, "p"] & \textbf{Spec}(B)\arrow[r, "x"] & X 
\end{tikzcd}
\]

Thus a map in $ \Hom_{B-\Mod} \left( y^*\Ll_{\faktor{\Aa (\_F)}{X} }, M \right)$ is completely determined by a map 
\[ \Phi : \textbf{Spec}(B \oplus M) \rightarrow \Aa (\_F)\] making the following diagram commute: 

\[ \begin{tikzcd}
\textbf{Spec} (B) \arrow[d, "i"] \arrow[dr, "y"] & \\
\textbf{Spec}(B \oplus M) \arrow[r, "\Phi"] \arrow[d, "p"] &  \Aa( \_F) \arrow[d, "\pi_X"] \\
\textbf{Spec}(B) \arrow[r, "x"] & X
\end{tikzcd}\]

Thus, we obtain that $ \Hom_{B-\Mod} \left( y^*\Ll_{\faktor{\Aa (\_F)}{X} }, M \right)$ is equivalent to 
\[ \tx{hofiber}_{s_y} \left(  \Map_{B \oplus M - \Mod} \left( B \oplus M, p^* x^* \_F \right) \rightarrow \Map_{B  - \Mod} \left( B, x^* \_F \right)\right)\]

where $s_y \in \Map_{B  - \Mod} \left( B, x^* \_F \right)$ is the section associated to $y : \textbf{Spec}(B) \rightarrow \Aa (\_F )$. The map is then given by precomposition with $i^*$. We can now observe that $p^* x^* \_F = x^* \_F \oplus x^* \_F \otimes_B M $ and that

\[ \Map_{B \oplus M - \Mod} \left( B \oplus M, p^* x^* \_F \right) \simeq \Map_{B  - \Mod} \left( B , x^* \_F \oplus x^* \_F \otimes_B M \right) \]

We obtain  
\[ \Hom_{B-\Mod} \left( y^*\Ll_{\faktor{\Aa (\_F)}{X} }, M \right)\] 
\[ \simeq \tx{hofiber} \left(  \Map_{B - \Mod} \left( B , x^* \_F \oplus x^* \_F \otimes_B M \right) \rightarrow \Map_{B  - \Mod} \left( B, x^* \_F \right)\right)\]
\[ \simeq \Map_{B  - \Mod} \left( B, x^* \_F \otimes_B M \right) \simeq \Map_{B  - \Mod} \left( x^* \_F^\vee, M \right) \]

Now the result follows from the fact that the functor 

\[ \begin{tikzcd}[row sep=tiny]
B- \Mod \arrow[r] & \tx{Fun}\left( B-\Mod^{\leq 0}, \tx{sSet} \right) \\
N \arrow[r, mapsto] & \Map_{B-\Mod} \left( N, \bullet \right)
\end{tikzcd} \]\\

is fully faithful and the fact that everything we did is natural in $B$. 
\end{proof}

\begin{Lem}
\label{Lem_RelativeCotangentBaseRingChange}
Let $f: X \rightarrow Y$ be a morphism of derived Artin stacks. We consider $\_F \in \textbf{QC}(Y)$ dualisable. Then there is a commutative square: 
 
 \[ \begin{tikzcd}
  \Phi^* \Ll_{\faktor{\Aa ( \_F )}{X}} \arrow[r] \arrow[d, "\simeq"] &  \Ll_{\faktor{\Aa ( f^*\_F)}{Y}} \arrow[d, "\simeq"] \\
  \Phi^* \pi_Y^*\_F^\vee \arrow[r, "\simeq"] & \pi_X^* f^* \_F^\vee \end{tikzcd} \]

  with $\Phi$ the natural morphism in the following homotopy pull-back:
  
  \[ \begin{tikzcd}
  \Aa (f^* \_F) \simeq f^*\Aa(\_F) \arrow[r, "\Phi"] \arrow[d, "\pi_X"] & \Aa (\_F) \arrow[d, "\pi_Y"] \\
  X \arrow[r, "f"] & Y
  \end{tikzcd} \]
  
  and the lower horizontal equivalence $\Phi^* \pi_Y^* \_F^\vee \rightarrow \pi_X^* f^* \_F^\vee$ being the equivalence coming from the fact that $\pi_Y \circ \Phi \simeq f \circ \pi_X$.  
\end{Lem} 
\begin{proof}
The first things we observe is that $\Aa (f^* \_F) \simeq f^* \Aa(\_F)$. We consider as before $B$-points: 

\[ \begin{tikzcd}
\textbf{Spec}(B) \arrow[r, "y"] \arrow[rr, bend left, "\tilde{y}"] \arrow[dr, "x"] \arrow[drr, bend right, "\tilde{x}"' ] & f^*\Aa( \_F) \arrow[r, "\Phi"] \arrow[d, "\pi_X"] & \Aa (\_F) \arrow[d, "\pi_Y"] \\
& X \arrow[r, "f"] & Y
\end{tikzcd} \]

We want to show that the following diagram is commutative: 

\begin{equation}
\label{Dia_1}
 \begin{tikzcd}
\Hom_{B-\Mod} \left( y^* \Ll_{\faktor{ \Aa (f^*\_F)}{X}}, M \right) \arrow[r] \arrow[d, "\simeq"] & \Hom_{B-\Mod} \left( \tilde{y}^* \Ll_{\faktor{ \Aa (\_F)}{Y}}, M \right) \arrow[d, "\simeq"] \\
\Hom_{B-\Mod} \left( y^* \pi_X^* f^* \_F^\vee, M \right)  \arrow[r] & \Hom_{B-\Mod} \left( \tilde{y}^* \pi_Y^* \_F^\vee, M \right)
\end{tikzcd}
 \end{equation}

Using the universal property of the cotangent complex, the top horizontal arrow is naturally equivalent to the map

\[ \begin{tikzcd}
\tx{hofiber}_{y} \left( \Hom_{\faktor{\textbf{dSt}}{X}} \left( \textbf{Spec}(B \oplus M), \Aa (f^* \_F) \right) \rightarrow  \Hom_{\faktor{\textbf{dSt}}{X}} \left( \textbf{Spec}(B), \Aa (f^* \_F) \right) \right) \arrow[d] \\
\tx{hofiber}_{\tilde{y}} \left( \Hom_{\faktor{\textbf{dSt}}{Y}} \left( \textbf{Spec}(B \oplus M), \Aa (\_F) \right) \rightarrow  \Hom_{\faktor{\textbf{dSt}}{Y}} \left( \textbf{Spec}(B), \Aa (\_F) \right) \right) 
\end{tikzcd} \]

induced by $\Hom_{\textbf{dSt}}\left( -, \Phi \right)$. A map $\psi : \textbf{Spec}(B \oplus M) \rightarrow \Aa(f^* \_F)$ in this homotopy fiber fits in the following commutative diagram:

\[ \begin{tikzcd}
\textbf{Spec}(B) \arrow[rd, "y"'] \arrow[rrd, "\tilde{y}"] \arrow[d, "i"] & & \\
\textbf{Spec}(B \oplus M) \arrow[d, "p"] \arrow[r, "\psi"'] & \Aa (f^* \_F ) \arrow[d, " \pi_X" ] \arrow[r, "\Phi"'] & \Aa (\_F) \arrow[d, "\pi_Y"] \\
\textbf{Spec}(B) \arrow[r, "x"] & X \arrow[r, "f"] & Y
\end{tikzcd} \] 

and the map between the homotopy fiber sends $\psi$ to $\Phi \circ \psi$. Since the underlying map of $\psi$ is $\pi_X \circ \psi : \textbf{Spec}(B \oplus M) \rightarrow X$ is $x \circ p$ and the underlying map of $\Phi \circ \psi$ is $\pi_Y \circ \Phi \circ \psi : \textbf{Spec}(B \oplus M) \rightarrow Y$ is $f\circ x \circ p = \tilde{x} \circ p$, this map between the homotopy fiber of derived stacks is therefore naturally equivalent to the map: 

\[ \begin{tikzcd}
\tx{hofiber}_{s_y} \left( \Map_{B \oplus M- \Mod} \left( B \oplus M, p^*x^* f^* \_F) \right) \rightarrow  \Hom_{B-\Mod} \left( B,  p^*x^* f^* \_F \right) \right) \arrow[d] \\
\tx{hofiber}_{s_{\tilde{y}}} \left( \Map_{B \oplus M- \Mod} \left( B \oplus M, p^*\tilde{x}^* \_F) \right) \rightarrow  \Hom_{B-\Mod} \left( B,  p^* \tilde{x}^* \_F \right) \right)
\end{tikzcd} \]

where $s_y $ and $s_{\tilde{y}}$ are the sections associated to $y$ and $\tilde{y}$ respectively. This map is in fact induced by the natural identification $ p^* \tilde{x}^* \_F \simeq  p^*x^* f^* \_F$ (since $\tilde{x} = f\circ x$). But following the steps of the proof of Proposition \ref{Prop_RelativeCotangentComplexForLinearStacks}, this map is naturally equivalent to the map 

\[ \Hom_{B-  \Mod} \left( y^* \pi_X^* f^*  \_F^\vee , M \right) \rightarrow \Hom_{B-  \Mod} \left( \tilde{y}^* \pi_Y^* \_F^\vee, M \right)\]

The natural equivalences we used are the natural equivalences used in the proof of Proposition \ref{Prop_RelativeCotangentComplexForLinearStacks} which proves that the Diagram \eqref{Dia_1} is commutative. Now the result follows once again from the fact that the functor 

\[ \begin{tikzcd}[row sep=tiny]
B- \Mod \arrow[r] & \tx{Fun}\left( B-\Mod^{\leq 0}, \tx{sSet} \right) \\
N \arrow[r, mapsto] & \Map_{B-\Mod} \left( N, \bullet \right)
\end{tikzcd} \]\\

is fully faithful and the fact that everything we did is natural in $B$. 
\end{proof}

\begin{Lem}
\label{Lem_RelativeCotangentNaturalityInFiber}
Let $X$ be a derived Artin stacks. We consider $\_F, \_G \in \textbf{QC}(X)$ dualisable and $h : \_F \rightarrow \_G$. Then there is a commutative square: 
 
 \[ \begin{tikzcd}
  \hat{h}^*\Ll_{\faktor{\Aa (\_G )}{X}} \arrow[r] \arrow[d, "\simeq"] &  \Ll_{\faktor{\Aa ( \_F)}{X}} \arrow[d, "\simeq"] \\
  \pi_X^* \_G^\vee \arrow[r, "\pi_X^* h^\vee"] &  \pi_X^* \_F^\vee \end{tikzcd} \]
  
  with $\hat{h} : \Aa ( \_G) \rightarrow \Aa (\_F)$ the map induced by $\_F$. 
\end{Lem} 

\begin{proof}
Every step of the proof of Proposition \ref{Prop_RelativeCotangentComplexForLinearStacks} is functorial in $\_F$.
\end{proof}

\begin{Prop}

\label{Prop_RelativeCotangentComplexFunctorialityForLinearStacks} 
 Let $f: X \rightarrow Y$ be a morphism of derived Artin stacks. We consider $\_F \in \textbf{QC}(X)$ and $\_G \in \textbf{QC}(Y)$ dualisable and a morphism $h : f^* \_F \rightarrow \_G$. Then there is a commutative square: 
 
 \[ \begin{tikzcd}
  \Ll_{\faktor{\Aa ( \_F )}{X}} \arrow[r] \arrow[d, "\simeq"] & \hat{f}^* \Ll_{\faktor{\Aa ( \_G )}{Y}} \arrow[d, "\simeq"] \\
  \pi_X^* \_F^\vee \arrow[r, "\pi_X^* h^\vee"] & \pi_X^* f^* \_G^\vee = \hat{f}^* \pi_Y^* \_G^\vee
 \end{tikzcd} \]
 \end{Prop}
  \begin{proof}
 It follows from Lemma \ref{Lem_RelativeCotangentBaseRingChange} and Lemma \ref{Lem_RelativeCotangentNaturalityInFiber}.
 \end{proof}

 \begin{Prop}
 \label{Prop_CotangentNonDegeneracyFibrationIsNatural}
 The quasi-isomorphism $\alpha_{\pi_X} : \Tt_{\faktor{T^*[n]X}{X}} \rightarrow \pi_X^* \Ll_X[n]$ of Example \ref{Ex_CotangentLagrangianFibration} expressing the non-degeneracy of the canonical Lagrangian fibration on the shifted cotangent stacks is the canonical quasi-isomorphism from Proposition \ref{Prop_RelativeCotangentComplexForLinearStacks}.
 \end{Prop}

\begin{proof}
First, since the cotangent bundle has a section, we have a split exact sequence: 

\[ \begin{tikzcd}
\pi_X^* \Ll_X [n] \arrow[r, shift left] & \arrow[l, dashed, shift left] \Ll_{T^*[n]X} [n] \arrow[r, shift left] & \arrow[l, dashed, shift left]  \Ll_{\faktor{T^*[n]X}{X}}[n] 
\end{tikzcd} \]

Proposition \ref{Prop_RelativeCotangentComplexForLinearStacks} gives us canonical equivalences $ \Ll_{\faktor{T^*[n]X}{X}}[n] \simeq \pi_X^* \Tt_X$. With this data, we can rewrite Diagram \eqref{Dia_LagrangianFibrationNonDegeneracy}, up to weak equivalences, as the strictly commutative diagram

\[ \begin{tikzcd}
\Tt_{\faktor{T^*[n]X}{X}} \arrow[r, "\simeq"] \arrow[d] & \pi_X^* \Ll_X[n] \arrow[d] \arrow[r] & 0 \arrow[d] \\
\pi_X^* \Ll_X[n] \oplus \pi_X^* \Tt_X \arrow[r, "\omega"] & \pi_X^* \Tt_X \oplus \pi_X^* \Ll_X[n] \arrow[r] & \pi_X^* \Tt_X
\end{tikzcd}\]

Through the canonical equivalence $\Tt_{\faktor{T^*[n]X}{X}} \rightarrow \pi_X^* \Ll_X[n]$ of Proposition \ref{Prop_RelativeCotangentComplexForLinearStacks}, the morphism $\Tt_{\faktor{T^*[n]X}{X}} \rightarrow \pi_X^* \Ll_X[n]\oplus \pi_X^* \Tt_X$ simply becomes the natural inclusion and $\omega$ becomes the identity. This implies that $ \alpha_{\pi_X} : \Tt_{\faktor{T^*[n]X}{X}} \rightarrow \pi_X^* \Ll_X[n]$ is the canonical equivalence of Proposition \ref{Prop_RelativeCotangentComplexForLinearStacks}.
\end{proof} 
 \section{Symplectic Geometry of the Derived Critical Locus}
\label{Sec_SymplecticGeometryoftheDerivedCriticalLocus} 
 
  In \ref{Sec_LagrangianIntersection} and \ref{Sec_LagragianFibIntersection} we present a few results on the symplectic geometry of homotopy pull-backs of derived Artin stacks. These results apply in particular to the case of derived intersections of derived schemes. In Section \ref{Sec_DerivedCriticalLocus} we study in more details the special case of derived intersections given by derived critical loci.

 \subsection{Lagrangian Intersections are $(n-1)$-Shifted Symplectic}
 \label{Sec_LagrangianIntersection}
 
 \begin{Prop}[\cite{PTVV}, Section 2.2]
\label{Prop_LagrangianIntersectionareShiftedSymplectic}
Let $Z$ be a derived Artin stack with a $n$-shifted symplectic structure $\omega$.  Let $f: X \rightarrow Z$ and $g : Y \rightarrow Z$ be morphisms with $\gamma$ and $\delta$ Lagrangian structures on $f$ and $g$ respectively. Then the homotopy pull-back $X \times_Z Y$ possesses a canonical $(n-1)$-shited symplectic structure called the residue of $\omega$ and denoted $R(\omega, \gamma,\delta)$. 
\end{Prop}

\begin{RQ}
	\label{RQ_LagrangianIntersectionMap}
	If we fix $f$ and $g$ as above, we can extend the previous theorem to obtain the following map of spaces (see Theorem 2.4 in \cite{Cal4}):
	\[ \textbf{Lag}(f,n) \times_{\textbf{Symp}(X,n)} \textbf{Lag}(g,n) \rightarrow \textbf{Symp}(X\times_Z Y, n-1)\]

\end{RQ}

\begin{RQ}
Theorem \ref{Prop_LagrangianIntersectionareShiftedSymplectic} can also be seen as a consequence of the procedure of composition of Lagrangian correspondences. Consider the following composition of Lagrangian correspondences: 

\[ \begin{tikzcd}
 & & X \times_Z Y \arrow[dl] \arrow[dr] & & \\
 & X \arrow[dl] \arrow[dr]& & Y \arrow[dl] \arrow[dr] & \\
\star & & Z  & & \star 
\end{tikzcd}\]

The maps $X \rightarrow Z \times \bar{\star}$ and $Y \rightarrow Y \times \bar{\star}$ are Lagrangian correspondences because $X \rightarrow Z$ and $Y \rightarrow Z$ are Lagrangian. Therefore, by composition, $X \times_Z Y \rightarrow \star \times \bar{\star} $ is also a Lagrangian correspondence, thus $X \times_Z Y \rightarrow \star$ is Lagrangian. From Lemma \ref{Lem_LagrangianOveraPoint}, since the point is $n$-shifted symplectic, then $X \times_Z Y$ is $(n-1)$-shifted symplectic.
\end{RQ}
  
 \subsection{Lagrangian Fibrations and Derived Intersections}

\label{Sec_LagragianFibIntersection}

\begin{Prop}
\label{Prop_LagFibProp1}

Suppose we have a sequence $\begin{tikzcd}
  L \arrow[r, "f"]&  Y \arrow[r,"g"]& X
\end{tikzcd}$ of Artin stacks and $\omega$ a $n$-shifted symplectic form on Y. Assume that $f $ is a Lagrangian morphism and $g$ is a Lagrangian fibration. Then there is a canonical quasi-isomorphism $\Tt_{\faktor{L}{X}} \rightarrow \Ll_{\faktor{L}{X}}[n-1]$. \\
\end{Prop}

\begin{proof}

Consider the following commutative diagram: \\

\adjustbox{scale=0.8,center}{
\begin{tikzcd}[row sep=0.8cm, column sep=0.5cm]
\Tt_L \arrow[dd, "\simeq", near end] \arrow[rr] & & f^* \Tt_Y \arrow[dd, "\simeq", near end] \arrow[rr] & & (g\circ f )^* \Tt_X \arrow[dd, dashed, "\simeq", near end] & \\
 & \Tt_{\faktor{L}{X}} \arrow[dd, dashed,crossing over,  "\simeq", near start] \arrow[lu] \arrow[rr, crossing over]& & f^* \Tt_{\faktor{Y}{X}} \arrow[dd,crossing over,  "\simeq", near start] \arrow[lu] \arrow[rr, crossing over] & & 0 \arrow[dd, crossing over, "\simeq"] \arrow[ul] \\
 \Ll_{\faktor{L}{Y}}[n-1] \arrow[rr] & & f^* \Ll_Y [n] \arrow[rr] & & f^*\Ll_{\faktor{Y}{X}}[n] & \\
 &\Ll_{\faktor{L}{X}}[n-1] \arrow[ul] \arrow[rr] & &  (g \circ f)^* \Ll_X [n] \arrow[ul] \arrow[rr] & & 0 \arrow[ul]
\end{tikzcd}
}\\

In the upper face, every squares are bi-Cartesian because both the outer square and the right most square are bi-Cartesian. Every non-dashed vertical arrows are quasi-isomorphisms by assumption (because of the various non-degeneracy conditions). Focusing on the right hand cube, it sends the upper homotopy bi-Cartesian square to the bottom square which is also homotopy bi-Cartesian. The homotopy cofiber of $(g \circ f)^* \Ll_X [n] \rightarrow f^* \Ll_Y[n]$ is $f^* \Ll_{\faktor{Y}{X}} [n]$ and we obtain a quasi-isomorphism $(g \circ f)^* \Tt_X \rightarrow f^* \Ll_{\faktor{Y}{X}}[n]$ depicted as a dashed arrow. 

By the same reasoning, since the upper outer square is homotopy bi-Cartesian, it maps to the lower outer square who is also homotopy bi-Cartesian. Moreover, the homotopy fiber of the map $\Ll_{\faktor{L}{Y}} [n-1] \rightarrow f^* \Ll_{\faktor{Y}{X}}[n]$ is exactly $\Ll_{\faktor{L}{X}}[n-1]$. This proves that there is a canonical quasi-isomorphism $\Tt_{\faktor{L}{X}} \rightarrow \Ll_{\faktor{L}{X}} [n-1]$. 
\end{proof}

\begin{Th}
\label{Th_DerivedIntersectionLagrangianFibration}
Let $Y$ be a $n$-shifted symplectic derived Artin stack. Let $f_i: L_i \rightarrow Y$ be Lagrangian morphisms (for $i= 1 \cdots 2$) and $\pi : Y \rightarrow X$ a Lagrangian fibration. Suppose that the maps $\pi \circ f_i : L_i \rightarrow X$ are weak equivalences. Then $P : Z = L_1 \times_Y L_2 \rightarrow X$ is a Lagrangian fibration.
\end{Th}

\begin{proof}

We summarize the notation in the following diagram:

$$\begin{tikzcd}
Z \arrow[r, "p_1"] \arrow[dr, "F"] \arrow[d, "p_2"] & L_1 \arrow[d, "f_1"]  & \\
L_2 \arrow[r,"f_2"] & Y \arrow[dr, " \pi"] & \\
& & X 
\end{tikzcd}
$$

We also denote $P := \pi \circ F : Z \rightarrow X $. \\

To show that we can obtain an isotropic structure, we will show that we have a map of spaces (dropping at first the non-degeneracy condition of the Lagrangian fibration): 
\[ \textbf{Lag}(f_1,n) \times_{\textbf{Symp}(Y,n)} \textbf{Lag}(f_2,n) \times_{\textbf{Symp}(Y,n) } \textbf{IsoFib}(\pi, n) \rightarrow \textbf{IsoFib}(P,n-1)  \]

If we forget the non-degeneracy of the Lagrangian structure, we obtain an element in $ \textbf{Iso}(f_1,n) \times_{\textbf{Symp}(Y,n)} \textbf{Iso}(f_2,n) \times_{\textbf{Symp}(Y,n) } \textbf{IsoFib}(\pi, n)$ and we can show, with formal manipulations of the pullbacks defining the spaces of isotropic structures and isotropic fibrations, that:
\[  \textbf{Iso}(f_1,n) \times_{\textbf{Symp}(Y,n)} \textbf{Iso}(f_2,n) \times_{\textbf{Symp}(Y,n) } \textbf{IsoFib}(\pi, n)\] 
\[ = \star \times_{\_A^{2,cl}(L_1,n)} \textbf{Symp}(Y,n) \times_{\_A^{2,cl} (L_2,n)} \star  \times_{\_A^{2,cl}(Y/X,n)} \star \]

Using the pullback to $\_A^{2,cl}(L_1 \times_Y L_2, n)$ we obtain a morphism: 
\[  \textbf{Iso}(f_1,n) \times_{\textbf{Symp}(Y,n)} \textbf{Iso}(f_2,n) \times_{\textbf{Symp}(Y,n) } \textbf{IsoFib}(\pi, n) \rightarrow\] 
\[  \star \times_{\_A^{2,cl}(L_1 \times_Y L_2,n)} \_A^{2,cl}(L_1 \times_Y L_2,n) \times_{\_A^{2,cl} (L_1 \times_Y L_2,n)} \star \times_{\_A^{2,cl}(Y/X,n)} \star\]

This last space naturally maps to: 
\begin{itemize}
	\item  \[ \_A^{2,cl}(L_1 \times_Y L_2, n-1) = \star \times_{\_A^{2,cl}(L_1 \times_Y L_2,n)} \_A^{2,cl}(L_1 \times_Y L_2,n) \times_{\_A^{2,cl} (L_1 \times_Y L_2,n)} \star \] Moreover, if we restrict this map to non-degenerate isotropic structures, then it is valued in $\textbf{Symp}(L_1 \times_Y L_2, n-1)$ (thanks to Theorem \ref{Prop_LagrangianIntersectionareShiftedSymplectic}).
	\item \[  \_A^{2,cl}(Y/X,n-1) = \star \times_{\_A^{2,cl}(Y/X,n)} \star \]
\end{itemize}

We have the commutative diagram:
\[ \begin{tikzcd}
	\_A^{2,cl}(L_1 \times_Y L_2, n-1) \times \_A^{2,cl}(Y/X,n-1) \arrow[r] \arrow[d] & \_A^{2,cl}(Y/X,n-1) \arrow[d, "P^*"] \\
	 \_A^{2,cl}(L_1 \times_Y L_2, n-1) \arrow[r] &  \_A^{2,cl}(L_1 \times_Y L_2/X, n-1)
\end{tikzcd}\]

Since the map $\_A^{2,cl}(Y/X,n-1)  \rightarrow \_A^{2,cl}(L_1 \times_Y L_2/X, n-1)$ factors through $\_A^{2,cl} (L_i/X, n-1) \simeq \star$, we get a morphism: 
\[ 	\_A^{2,cl}(L_1 \times_Y L_2, n-1) \times \_A^{2,cl}(Y/X,n-1) \rightarrow \_A^{2,cl}(L_1 \times_Y L_2, n-1) \times_{\_A^{2,cl}(L_1 \times_Y L_2/X,n-1)} \star \]

Now if we restrict to $\textbf{Symp}(L_1 \times_Y L_2, n-1) \subset \_A^{2,cl}(L_1 \times_Y L_2, n-1)$ (which amounts to restricting to non-degenerate isotropic structures), we get a map: 
\[ \textbf{Symp}(L_1 \times_Y L_2, n-1) \times \_A^{2,cl}(Y/X,n-1) \rightarrow \textbf{IsoFib}(P, n-1) \]

Therefore we get the desired map and we will consider the isotropic fibration on $P$ given by the image along the morphism we just described of the Lagrangian structures and Lagrangian fibration structure given on $f_1$, $f_2$ and $\pi$ respectively. We are left to prove the non-degeneracy condition. To do that, we first consider the following diagram: 
 
\[ \begin{tikzcd}
	\Tt_{\faktor{Z}{X}} \arrow[d] \arrow[r] &   p_1^* \Tt_{\faktor{L_1}{X}} \oplus p_2^* \Tt_{\faktor{L_2}{X}} \arrow[d] \arrow[r] & F^*\Tt_{\faktor{Y}{X}} \arrow[d] \\
	\Tt_{Z} \arrow[d] \arrow[r] &   p_1^* \Tt_{L_1} \oplus p_2^* \Tt_{L_2} \arrow[d] \arrow[r] & F^*\Tt_{Y} \arrow[d] \\
	 P^*\Tt_X \arrow[r] &  P^*\Tt_{X} \oplus  P^*\Tt_{X}  \arrow[r] & P^*\Tt_{X} 
\end{tikzcd} \]

The vertical sequences and the last two horizontal sequences are fibered and therefore so is the first horizontal sequence. The last two horizontal sequences are fibered because the following diagrams are Cartesian: 
\[ \begin{tikzcd}
	\Tt_Z \arrow[r] \arrow[d] &  p_1^* \Tt_{L_1} \arrow[d]  \\
	p_2^* \Tt_{L_2} \arrow[r] &  F^*\Tt_{Y}
\end{tikzcd} \qquad
\begin{tikzcd}
	\Tt_X \arrow[r] \arrow[d] & \Tt_X \arrow[d]  \\
\Tt_X \arrow[r] &  \Tt_X
\end{tikzcd}
\] 

 Using Proposition \ref{Prop_LagFibProp1} and non-degeneracy, we get the following commutative diagram: 

\[\begin{tikzcd}
	\Tt_{\faktor{Z}{X}} \arrow[r] \arrow[d, dashed] & p_1^* \Tt_{\faktor{L_1}{X}} \oplus p_2^* \Tt_{\faktor{L_2}{X}} \arrow[d] \arrow[r] & F^*\Tt_{\faktor{Y}{X}} \arrow[d] \\
	P^*\Ll_X[n-1] \arrow[r] & \left(p_1^* \Tt_{\faktor{L_1}{X}} \oplus p_2^* \Tt_{\faktor{L_2}{X}}\right)[n-1] \arrow[r] & P^* \Ll_X[n]
\end{tikzcd}\]  

where all vertical morphisms are quasi-isomorphisms. The fiber in the lower sequence is exactly $\Ll_X[n-1] $ because $p_1^*\Tt_{\faktor{L_1}{X}} \oplus p_2^* \Tt_{\faktor{L_2}{X}} \simeq 0$ since $L_i \to X$ are equivalences. We will call $\alpha : 	\Tt_{\faktor{Z}{X}} \to 	P^*\Ll_X[n-1] $ the dashed equivalence obtained. 

  We still need to show that $\alpha$ is the morphism used in the criteria for the non-degeneracy of the Lagrangian fibration. Recall that this morphism is given by means of the universal map filling Diagram \eqref{Dia_LagrangianFibrationNonDegeneracy}:
 
 \begin{equation*}
\begin{tikzcd}
\Tt_{\faktor{Z}{X}} \arrow[d] \arrow[r, "\alpha_P"] & P^* \Ll_X[n-1] \arrow[d] \arrow[r] & 0 \arrow[d] \\
\Tt_Z \arrow[r, "\sim"] & \Ll_Z[n-1] \arrow[r] & \Ll_{\faktor{Z}{X}}[n-1] 
\end{tikzcd}
\end{equation*}

To compare $\alpha$ and $\alpha_P$, we summarize the construction of $\alpha$ and all the equivalences coming from non-degeneracy conditions in the following diagram: 

\begin{equation}
	\label{Dia_LagrangianFibrationDerivedIntersection}
	\adjustbox{scale=0.75,center}{
		\begin{tikzcd}[row sep=2cm, column sep=tiny]
			\Tt_Z \arrow[rr] \arrow[dd, "\simeq", near start] & & p_1^*\Tt_{L_1} \oplus p_2^* \Tt_{L_2} \arrow[rr] \arrow[dd, "\simeq", near start] & & F^*\Tt_Y  \arrow[dd, "\simeq", near start] &  \\
			&\Tt_{\faktor{Z}{X}} \arrow[rr, crossing over] \arrow[ul] \arrow[dd, dashed, "\simeq", crossing over,near start] & & p_1^* \Tt_{\faktor{L_1}{X}} \oplus p_2^* \Tt_{\faktor{L_2}{X}} \arrow[rr, crossing over] \arrow[ul] \arrow[dd, "\simeq", crossing over, near start]& & F^* \Tt_{\faktor{Y}{X}} \arrow[ul] \arrow[dd, "\simeq", crossing over, near start]\\
			\Ll_Z [n-1]  \arrow[rr] & & \left( p_1^* \Ll_{\faktor{L_1}{Y}} \oplus p_2^* \Ll_{\faktor{L_2}{Y}}\right) [n-1] \arrow[rr]& & F^* \Ll_Y[n]  & \\
			& P^*\Ll_X[n-1] \arrow[ul] \arrow[rr, crossing over]& &\left(  p_1^* \Ll_{\faktor{L_1}{X}} \oplus p_2^* \Ll_{\faktor{L_2}{X}}\right)[n-1] \arrow[ul] \arrow[rr, crossing over] & & P^*\Ll_X[n] \arrow[ul] 
		\end{tikzcd} 
	}\\
\end{equation}

where all the vertical maps are quasi-isomorphism obtained from the non-degeneracy conditions. We want to prove that $\alpha_P$ and  $\alpha$ are homotopic. The relevant data extracted from the Diagram \eqref{Dia_LagrangianFibrationDerivedIntersection} is:

 \begin{equation*}
\begin{tikzcd}
\Tt_{\faktor{Z}{X}} \arrow[d] \arrow[r, "\alpha"] & P^* \Ll_X[n-1] \arrow[d] \arrow[r] & 0 \arrow[d] \\
\Tt_Z \arrow[r, "\sim"] & \Ll_Z [n-1] \arrow[r] & p_1^* \Ll_{\faktor{L_1}{Y}}[n-1] \oplus  p_2^* \Ll_{\faktor{L_2}{Y}}[n-1]  
\end{tikzcd}
\end{equation*}

 The composition: 
 \[P^* \Ll_X[n-1] \rightarrow \Ll_Z[n-1] \rightarrow p_1^* \Ll_{\faktor{L_1}{Y}}[n-1] \oplus  p_2^* \Ll_{\faktor{L_2}{Y}}[n-1] \] factorizes through $0 \simeq p_1^* \Ll_{\faktor{L_1}{X}}[n-1] \oplus  p_2^* \Ll_{\faktor{L_2}{X}}[n-1]$. This implies that the map $ \Ll_Z[n-1] \rightarrow p_1^* \Ll_{\faktor{L_1}{Y}}[n-1] \oplus  p_2^* \Ll_{\faktor{L_2}{Y}}[n-1] $ factorizes through $\Ll_{\faktor{Z}{X}}[n-1]$ and therefore $\alpha$ satisfies the same universal property as $\alpha_P$, proving that $\alpha$ and $\alpha_P$ are homotopic.   
\end{proof} 

\begin{RQ}
	Similarly to Proposition \ref{Prop_LagrangianIntersectionareShiftedSymplectic}, this theorem can be extended to a map of spaces:
	\[ \textbf{Lag}(f_1,n) \times_{\textbf{Symp}(Y,n)} \textbf{Lag}(f_2,n) \times_{\textbf{Symp}(Y,n)} \textbf{LagFib}(\pi,n) \rightarrow \textbf{LagFib}(P,n) \] 
	
		This is the simply the restriction the map described in the proof of Theorem \ref{Th_DerivedIntersectionLagrangianFibration} to the non-degenerate elements. Forgetting the extra Lagrangian fibration recovers the map in Remark \ref{RQ_LagrangianIntersectionMap}, that is the following diagram is commutative: 
	
	\[  \begin{tikzcd}
	\textbf{Lag}(f_1,n) \times_{\textbf{Symp}(Y,n)} \textbf{Lag}(f_2,n) \times_{\textbf{Symp}(Y,n)} \textbf{LagFib}(\pi,n)  \arrow[r] \arrow[d] & \textbf{LagFib}(P,n)  \arrow[d] \\
	\textbf{Lag}(f_1,n) \times_{\textbf{Symp}(Y,n)} \textbf{Lag}(f_2,n) \arrow[r] & \textbf{Symp}(L_1 \times_Y L_2, n-1)
	\end{tikzcd}  \]

\end{RQ}

 \subsection{Derived Critical Locus}
\label{Sec_DerivedCriticalLocus}

 Given a derived Artin stack $X$ and a morphism $f : X \rightarrow \Aa_k^1$, we define the derived critical locus of $f$, denoted $\textbf{Crit}(f)$, as the derived intersection of $df : X \rightarrow T^*X$ with the zero section $0 : X \rightarrow  T^*X$. It is given by the homotopy pull-back:
 
 \begin{equation}
 \label{Dia_DerivedCriticalLocus}
 \begin{tikzcd}
 \textbf{Crit}(f) \arrow[r] \arrow[d] & X \arrow[d, "df"] \\
 X \arrow[r,"0"] & T^*X
 \end{tikzcd}
 \end{equation}.
 
\begin{Ex}
\label{Ex_DerivedCriticalLocusSmoothAlgVariety}
We recall from \cite{Cal3} that if $X$ is a smooth algebraic variety, its derived critical locus can be described, as a derived scheme, by the underlying scheme given by the ordinary critical locus of $f$, that we denote $S$, together with the sheaf of $\textbf{cdga}_{\leq 0}$ given by the derived tensor product $\_O_X \otimes_{\_O_{T^*X}}^\Ll \_O_X$, restricted to $S$. This derived tensor product is described by the homotopy push-out:

$$ \begin{tikzcd}
\Sym_{\_O_X}\left(  \Tt_X \right) \arrow[r, "0"] \arrow[d, "df"] & \_O_X \arrow[d] \\
\_O_X \arrow[r] & \_O_X \otimes_{\_O_{T^*X}}^\Ll \_O_X
\end{tikzcd}
$$

 Taking the derived tensor product amounts to replacing the 0-section morphism $0 : \Sym_{\_O_X} \Tt_X  \rightarrow \_O_{X} $ by the equivalent cofibration, in the model category of commutative differential graded k-algebras, $\Sym_{\_O_X}\Tt_X \hookrightarrow \Sym_{\_O_X} \left( \Tt_X[1] \oplus  \Tt_X\right) $, where $\Sym_{\_O_X} \left( \Tt_X[1] \oplus  \Tt_X \right)$ has the differential induced by $\tx{Id}: \Tt_X[1] \rightarrow \Tt_X$. Then we take the strict push-out of this replacement. The use of these resolutions is well explained in \cite{Cal3} or \cite{Ve1}. We obtain:

 $$\_O_{\textbf{Crit}(f)} := \left( \_O_X \otimes_{\_O_{T^*X}}^\Ll \_O_X \right)_{\vert S} \simeq \left( \Sym_{\_O_X} \Tt_X[1], \iota_{df} \right)_{\vert S}$$  
 
 where $\iota_{df}$ is the the differential on $\_O_{\textbf{Crit}(f)}$ given by the contraction along $df$. The restriction to $S$ denotes the fact that this is a derived scheme whose underlying scheme is the strict critical locus. Observe that outside of the critical locus, $\left( \Sym_{\_O_X} \Tt_X[1], \iota_{df} \right)$ is cohomologically equivalent to 0.                             
\end{Ex} 

\begin{RQ}\label{RQ_DerivedCriticalLocusSmoothCase}
If we do not assume that $X$ is smooth in Example \ref{Ex_DerivedCriticalLocusSmoothAlgVariety}, then $\Ll_X$ usually has a non trivial internal differential. As a sheaf of graded algebra, we still obtain $ \Sym_{\_O_X} \left( \Tt_X[1] \right)$ since the replacement is the same as a graded algebra but the differential is a priori be different and involves a combination of the internal differential on $\Tt_X$ and the contraction $\iota_{df}$. 
\end{RQ}

\begin{RQ}
\label{RQ_DerivedCriticalLocus}
From Example \ref{Ex_CotangentComplexisSymplectic}, we know that $T^*X$ carries a canonical symplectic form of degree $0$ and from Example \ref{Ex_LagrangianSectionof1form} we know that both the $0$ section and $df$ have a natural Lagrangian structure. From Proposition \ref{Prop_LagrangianIntersectionareShiftedSymplectic}, the derived intersection of these Lagrangian structures, namely the derived critical locus \textbf{Crit}(f), has a natural $(-1)$-shifted symplectic structure. \\
\end{RQ}

\begin{RQ}
\label{RQ_SelfIntersection}
When $X$ is a derived Artin stack and $df=0$, we have that $\textbf{Crit}(f) \simeq T^*[-1]X$ and $\omega_{\textbf{Crit}(f)}$ is the canonical $(-1)$-shifted symplectic structure on $T^*[-1]X$. 

In this situation, the strict critical locus is $X$ itself, and the restriction to $X$ of $\Sym_{\_O_X} \Tt_X[1]$ is therefore $\Sym_{\_O_X} \Tt_X[1]$ itself (with the differential being induced by the differential on $\Tt_X$). Thus $\textbf{Crit}(f) \simeq \textbf{Spec}_X \left( \Sym_{\_O_X} \Tt_X[1] \right) = T^*[-1]X $. 
\end{RQ}

\begin{RQ}
We want to understand in general the $(-1)$-shifted symplectic form on $\textbf{Crit}(f)$. We use the universal property of the tautological $1$-form (Lemma \ref{Lem_UniversalPropTautological1Form}) to see that $(df)^* \omega = 0$ (with $\omega = d\lambda_X$ the canonical symplectic structure on $T^*X$). Using the resolution of the zero section, as in Example \ref{Ex_DerivedCriticalLocusSmoothAlgVariety}, $\omega$ induces a closed 2-form on $\textbf{Spec}_X \left( \Sym_{\_O_X} \left( \Tt_X[1] \oplus  \Tt_X\right)\right)$. Since the differential on the resolution, $\Sym_{\_O_X} \left( \Tt_X[1] \oplus  \Tt_X\right)$, is induced by $\tx{Id} : \Tt_X \rightarrow \Tt_X[1]$, the tautological 1-form $\omega_{-1}$ on $T^*[-1]X$ induces a closed 2-form on $\textbf{Spec}_X \left( \Sym_{\_O_X} \left( \Tt_X[1] \oplus  \Tt_X\right)\right)$ which is a homotopy between $\omega$ and $0$. We then have that the $(-1)$-shifted symplectic form is described by the self-homotopy of $0$ given by $\omega_{-1}$: 

\[\begin{tikzcd}
0 \arrow[r, "\omega_{-1}"] & p^*\omega = 0
\end{tikzcd}\]

The proof of Proposition \ref{Prop_LagrangianIntersectionareShiftedSymplectic} (see Theorem 2.9 in \cite{PTVV}) tells us that  $\omega_{-1}$ is the (-1)-shifted symplectic form on $\textbf{Crit}(f)$.
\end{RQ}

\begin{RQ}
\label{RQ_CriticalLocusLagragianFibration}
From Theorem \ref{Th_DerivedIntersectionLagrangianFibration}, we have that $\pi : \textbf{Crit}(f) \rightarrow X$ is a Lagrangian fibration. In the situation where $df =0$ and $X$ is smooth, this Lagrangian fibration coincides with the canonical Lagrangian fibration on $\pi_X : T^*[-1]X \rightarrow X$. In general, the morphism $\alpha_{\pi}$ controlling the non-degeneracy condition of the Lagrangian fibration (see Diagram \eqref{Dia_LagrangianFibrationNonDegeneracy}) is still natural in the sense given by the following proposition.
\end{RQ}

\begin{Prop}

\label{Prop_LagrangianFibrationCanonicalNonDegeneracyMap}
$\alpha_{\pi}$ is equivalent to the following composition of equivalences: 

\begin{equation}
	\label{Eq_NonDegeneracyLagFibDerivedCriticalLocus}
	 \begin{tikzcd}
		\Tt_{\faktor{\textbf{Crit}(f)}{X}} \arrow[r] & \Tt_{\faktor{X}{X}} \times_{\Tt_{\faktor{T^*X}{X}}} \Tt_{\faktor{X}{X}} \arrow[r] \simeq 0 \times_{\Tt_{\faktor{T^*X}{X}}} 0 \arrow[r, "0 \times_{\beta} 0"] &  0 \times_{\pi_X^* \Ll_{X}} 0 \simeq \pi^* \Ll_X[-1]
	\end{tikzcd}
\end{equation}

where $\beta$ is the dual of the canonical equivalence $ \Ll_{\faktor{T^*X}{X}} \simeq \pi_X^* \Ll_X$ of Proposition \ref{Prop_RelativeCotangentComplexForLinearStacks}. 
\end{Prop}
\begin{proof}
The strategy here is to express the Diagram \eqref{Dia_LagrangianFibrationNonDegeneracy} as a pull-back of the same type of diagrams. It reduces the problem to proving the same statement but for the projection $\pi_X : T^*X \rightarrow X$. But this Proposition is known for the Lagrangian fibration on the shifted cotangent stacks (this is a direct consequence of Proposition \ref{Prop_CotangentNonDegeneracyFibrationIsNatural}).\\

First we express $\Ll_{\textbf{Crit}(f)}[-1]$ as a pull-back above $\Ll_{T^*X}$. This can be done by observing that all squares in the following diagram are bi-Cartesians:

\[ \begin{tikzcd}  
\Ll_{\textbf{Crit}(f)} [-1] \arrow[r] \arrow[d] &\pi_1^* 0^* \Ll_{\faktor{T^*X}{X}}  \arrow[r] \arrow[d] & 0 \arrow[d] \\
\pi_2^* df^* \Ll_{\faktor{T^*X}{X}}  \arrow[r] \arrow[d] & \pi_1^* 0^* \Ll_{T^*X} \simeq \pi_2^* df^*\Ll_{T^*X}  \arrow[r] \arrow[d] & \pi_2^*\Ll_X \arrow[d] \\
0 \arrow[r] & \pi_1^* \Ll_X \arrow[r] & \Ll_{\textbf{Crit}(f)} 
\end{tikzcd} \]

Where $\pi_1$ and $\pi_2$ ar the natural projections $\textbf{Crit}(f) \to X$ given by the pullback Diagram. We write Diagram \eqref{Dia_LagrangianFibrationNonDegeneracy} for $\pi : \textbf{Crit}(f) \rightarrow X$ as: 

\[ \begin{tikzcd}
0 \times_{\Tt_{\faktor{T^*X}{X}}} 0 \arrow[r, "\alpha_{\pi} \simeq 0 \times_{\alpha_{\pi_X}} 0"] \arrow[d] & 0 \times_{\pi_X^* \Ll_X} 0 \arrow[d] \arrow[r] & 0 \arrow[d] \\
\Tt_X \times_{\Tt_{T^*X}} \Tt_X \arrow[r] & \Ll_{\faktor{T^*X}{X}} \times_{\Ll_{T^*X}}\Ll_{\faktor{T^*X}{X}} \arrow[r, " \tx{Id} \times_{\tx{pr}} \tx{Id}"] & \Ll_{\faktor{T^*X}{X}} \times_{\Ll_{\faktor{T^*X}{X}}} \Ll_{\faktor{T^*X}{X}}
\end{tikzcd}\]

In this diagram, pullbacks have been omitted to keep the diagram easy to read. We need to describe the morphism $\omega_{\textbf{Crit}(f)} : \Tt_X \times_{\Tt_{T^*X}} \Tt_X \rightarrow \Ll_{\faktor{T^*X}{X}} \times_{\Ll_{T^*X}}\Ll_{\faktor{T^*X}{X}}$. Recall from Remark \ref{RQ_NonDegeneracyLagrangian} and the proof of the non-degeneracy in Proposition \ref{Prop_LagrangianIntersectionareShiftedSymplectic} that $\omega_{\textbf{Crit}(f)}$ is  $\Theta_{df} \times_{\omega} \Theta_0$ where $\Theta_h : \Tt_X \rightarrow \Ll_{h}[-1] \simeq \Ll_{\faktor{X}{T^*X}}[-1] \simeq \Ll_{\faktor{T^*X}{X}}$ is the natural morphism expressing the non-degeneracy of the Lagrangian structure (see Definition \ref{Def_DerivedLagrangianStructure}).\\

Finally, Proposition \ref{Prop_CotangentNonDegeneracyFibrationIsNatural} shows that $\beta$ is the same as $\alpha_{\pi_X}$. This completes the proof.  
\end{proof}

\section{Examples}

\label{Sec_Examples}

\subsection{One Non-Degenerate Critical Point}
\label{Sec_1NonDegenerateCritPoint}
Let $X$ be a smooth algebraic variety over k and $f : X \rightarrow \Aa_k^1$ a map which is smooth everywhere except at a point $x \in X$ where there is a non degenerate critical point. The goal is to understand the Lagrangian fibration on $\textbf{Crit}(f) \rightarrow X$ and show that it is related to the Hessian quadratic form of $f$ at $x$. This section is a particular case of Section \ref{Sec_FamillyNonDegenerateCriticalLocus}, and we only sketch what is happening in this case. We will be making the statements more precise and give complete proofs in Section \ref{Sec_FamillyNonDegenerateCriticalLocus}.\\

The strict critical locus is $\star := \left( \star, \faktor{\_O_X}{I} \right)$ where $I$ is the ideal generated by the partial derivatives of $f$, $I = \langle df.v, v \in \Tt_X \rangle$. There is a natural morphism $\tilde{x} := \star \rightarrow \textbf{Crit}(f)$ such that the following diagram commutes:

$$ \begin{tikzcd}
\star_{(-1)} \arrow[rd, "x"'] \arrow[r,"\tilde{x}"] & \textbf{Crit}(f) \arrow[d,"\pi"] \\
& X
\end{tikzcd}$$

 The ideal generated by the partial derivatives is maximal and the partial derivatives form a regular sequence. This implies that $\tilde{x}$ is an equivalence. For more details, this is the analogue of Proposition \ref{Prop_NonDegeneracyAndPhiRelatioship}. We can even prove (in the general context of Section \ref{Sec_FamillyNonDegenerateCriticalLocus}) that $T^*[-1]S $ (where $S$ is the strict critical locus) is weakly equivalent to $\textbf{Crit}(f)$. This result will however not be needed for the general case.\\  
 
 Using Lemma \ref{Lem_LagragianFibrationFromAPoint}, the Lagrangian fibration induced on $\star_{(-1)} \rightarrow X$ gives us a closed 2-form in $\_A^{2,cl}\left(\faktor{\star}{X}, -2 \right)$ which induces a metric on $\Tt_x X$. The non-degeneracy of the symmetric bilinear form is equivalent to the non-degeneracy of the Lagrangian fibration, which says that the natural map $x^*\Tt_X \rightarrow x^* \Ll_X$ is a quasi-isomorphism. We will show that this metric is in fact characterised by the Hessian quadratic form of $f$ at the critical point.  \\

We want to describe the Lagrangian fibration obtained on $\star \rightarrow X$ by pulling back along $\tilde{x}$ the homotopy between $\faktor{\omega_{-1}}{X}$ and 0 in $\_A^{2,cl}\left( \faktor{\textbf{Crit}(f)}{X},-1 \right)$. We obtain a homotopy between 0 and itself in $\_A^{2,cl}\left( \faktor{\star}{X},-1 \right)$. We will relate the Hessian quadratic form with the map $\alpha_{x}$ defined to describe the non-degeneracy condition of Lagrangian fibrations (see Definition \ref{def: lagrangian fibration} and Diagram \eqref{Dia_LagrangianFibrationNonDegeneracy}). For $\textbf{Crit}(f)$ and $\star$ this diagram becomes respectively:

\begin{equation*}
\begin{tikzcd}
\Tt_{\faktor{\textbf{Crit}(f)}{X}} \arrow[d] \arrow[r,"\alpha_{\pi}"] & \pi^*\Ll_X[-1] \arrow[d] \arrow[r] & 0 \arrow[d] \\
\Tt_{\textbf{Crit}(f)} \arrow[r, "\omega_{-1}^{\flat}"] & \Ll_{\textbf{Crit}(f)}[-1] \arrow[r] & \Ll_{\faktor{\textbf{Crit}(f)}{X}}[-1] 
\end{tikzcd}
\end{equation*}

and 

\begin{equation*}
\begin{tikzcd}
\Tt_{\faktor{\star}{X}} \arrow[d] \arrow[r,"\alpha_{x}"] & x^* \Ll_X[-1] \arrow[d] \arrow[r] & 0 \arrow[d] \\
0 \arrow[r] & 0 \arrow[r] & \Ll_{\faktor{\star}{X}}[-1].
\end{tikzcd}
\end{equation*}

These two diagrams are supposed to represent the same Lagrangian fibration. We will pull-back along $\tilde{x}$ the diagram for $\textbf{Crit}(f)$ to the category of differential graded $k$-vector space (i.e. $\textbf{QC}(\star)$). We can compare $\alpha_x$ and $\alpha_{\pi}$ via the following commutative diagram:\\

\adjustbox{scale=0.75,center}{
\begin{tikzcd}
 &\tilde{x}^* \Tt_{\faktor{\textbf{Crit}(f)}{X}} \arrow[rr, "\alpha_{\pi}"] \arrow[dd]  & & x^* \Ll_X[-1] \arrow[rr] \arrow[dd] \arrow[dl, dashed] & & 0 \arrow[dl, equals] \arrow[dd]  \\
 \Tt_{\faktor{\star}{X}} \arrow[ur, "\sim"] \arrow[rr, crossing over, "\alpha_x", near end] \arrow[dd, crossing over] & & \Ll_{\faktor{\star}{X}}\simeq x^* \Ll_X[-1] \arrow[rr, crossing over] \arrow[dd, crossing over] & & 0 \arrow[dd] & \\
  &\tilde{x}^* \Tt_{\textbf{Crit}(f)}  \arrow[rr, "\omega_{-1}^{\flat}", near start] & & \tilde{x}^* \Ll_{\textbf{Crit}(f)}[-1] \arrow[rr] \arrow[dl, dashed, "\psi"] & & \tilde{x}^* \Ll_{\faktor{\textbf{Crit}(f)}{X}}[-1]  \arrow[dl, "\sim"]\\
  0 \arrow[ur, "\sim"] \arrow[rr] & & \Ll_{\faktor{\star}{X}} \oplus  \Ll_{\faktor{\star}{X}} [-1] \arrow[rr]  & & \Ll_{\faktor{\star}{X}}[-1] & 
\end{tikzcd}
}\\

We can now look at these morphisms in local étale coordinates around $x$. We denote by $X^i$ coordinates in $X$, $p_i$ a basis of $x^* \Tt_X$ and $\xi^i$ its associated shifted basis in $x^* \Tt_X[1]$. We also denote by $dX^i$ the dual basis of $p_i$.  We write $k\langle a \rangle :=  k\langle a_1, \cdots, a_n\rangle$ for the k-vector space with basis $a_1, \cdots, a_n$. We get:\\

\adjustbox{scale=0.85,center}{
\begin{tikzcd}
 &k\langle \partial_\xi \rangle \arrow[rr, "\alpha_{\pi}"] \arrow[dd]  & &k \langle dX \rangle  \arrow[rr] \arrow[dd] \arrow[dl, dashed] & & 0 \arrow[dl, equals] \arrow[dd]  \\
 k\langle \partial_\xi \rangle  \arrow[ur, "\tx{Id}"] \arrow[rr, crossing over, "\alpha_x", near end] \arrow[dd, crossing over] & & k\langle d\theta \rangle \arrow[rr, crossing over] \arrow[dd, crossing over] & & 0 \arrow[dd] & \\
  &k\langle \partial_\xi , \partial_X \rangle  \arrow[rr, "\omega_{-1}^{\flat}", near start] & & k \langle dX, d\xi \rangle  \arrow[rr] \arrow[dl, dashed, "\psi"] & & k \langle d\xi \rangle  \arrow[dl, "\sim"]\\
  0 \arrow[ur, "\sim"] \arrow[rr] & & k \langle d\theta, d \xi \rangle \arrow[rr]  & & k \langle d\xi \rangle & 
\end{tikzcd}
}\\

Here, $d\theta$ is the standard shifted variable added to make the following pull-back square a strict pull-back: 

\[ \begin{tikzcd}
 k \langle d \theta \rangle \arrow[d] \arrow[r] & 0 \arrow[d] \\
 k\langle d \theta, d \xi \rangle \arrow[r] & k \langle d \xi \rangle 
 \end{tikzcd} \]

This imposes $\delta d \xi = d \theta$. To make the full diagram strictly commutative, we must have $\psi (d\xi) =  d\xi$. And to make $\psi$ a map of chain complexes, we must have $\psi(d \delta \xi^i) = \delta \psi(d \xi^i) = \delta d\xi^i = d \theta^i $ and therefore it imposes $\psi (dX^i) = \tx{Hess}_x^{-1}(f)(d X^i)(dX^j) d\theta^j$. This implies that $\alpha_x (\partial_{\xi^i}) =\tx{Hess}_x^{-1}(f)(d X^i)(dX^j) d\theta^j$.

\subsection{Family of non-Degenerate Critical Points}
\label{Sec_FamillyNonDegenerateCriticalLocus}
We consider a generalisation of the previous example where $f$ may have a family of critical points which are all non-degenerate in the directions normal to the critical locus. \\

Let us fix some notations. We denote by $S$ the strict critical locus, which comes with a closed immersion $i: S \rightarrow X$ and whose algebra of functions is $\_O_S = \faktor{\_O_X}{I}$ with $I = \langle df.v, \ v \in \Tt_X \rangle$. \\

We assume that both $X$ and $S$ are smooth algebraic varieties. We denote by $\textbf{Crit}(f)$ the derived critical locus of $f$ and we get a canonical morphism $\lambda : S \rightarrow \textbf{Crit}(f)$. \\

In order to define the Hessian quadratic form and the non-degeneracy condition, we need to assume that the closed immersion $S \hookrightarrow X$ has a first order splitting. Concretely, we assume all along in this Section that the following fiber sequence splits:

\begin{equation}\label{SplittingAssumption}
\begin{tikzcd}
\Tt_S \arrow[r, shift left] & i^* \Tt_X \arrow[l, shift left, dashed] \arrow[r, shift left] & \Tt_{\faktor{S}{X}}[1]\arrow[l, shift left, dashed]
\end{tikzcd}
\end{equation}  

This assumption is necessary to be able to restrict $Q$ to the normal part $\Tt_{\faktor{S}{X}}[1] $.

\begin{Def}\
The \defi{Hessian quadratic form} is defined by the symmetric bilinear map:

$$\begin{array}{ccccl}
Q & : &  \Sym_{\_O_S}^2 i^* \Tt_X  & \to & \_O_S \\
 & & (w,v) & \mapsto & d(df.v).w \\
\end{array}$$

 We define non-degeneracy to be along the "normal" direction to $S$, by considering the following diagram:  

\begin{equation}
\label{Dia_NonDegeneracyHessian}
\begin{tikzcd}
\Tt_S \arrow[r, shift left] \arrow[d, "0"] &\arrow[l, shift left, dashed] i^* \Tt_X \arrow[r, shift left] \arrow[d,crossing over, "Q", near end] &\arrow[l, shift left, dashed] \Tt_{\faktor{S}{X}}[1] \arrow[d, "0"] \arrow[dll,  dashed, "\widetilde{Q}", near start] \\
\Ll_{\faktor{S}{X}} [-1] \arrow[r, shift left] &\arrow[l, shift left, dashed] i^* \Ll_X \arrow[r, shift left] &\arrow[l, shift left, dashed] \Ll_S 
\end{tikzcd}
\end{equation}

Both rows are split fiber sequences (by assumption in Diagram \eqref{SplittingAssumption}). The left and right vertical maps are the zero map because $Q$ restricted to $\Tt_S$ is zero and, since $Q$ is symmetric, $Q$ composed with the projection to $\Ll_S$ is also zero.  We obtain a map $\widetilde{Q}$ (using $Q$ and following the section and retract of the fiber sequences) which corresponds to the map induced by $Q$ on the normal bundle. Then the \defi{non-degeneracy condition} is the requirement that $\widetilde{Q}$ is a quasi-isomorphism. 

\end{Def}

Since the differential on $\_O_{\textbf{Crit}(f)}$ is $ \delta = \iota_{df}$ (see Remark \ref{RQ_DerivedCriticalLocusSmoothCase}), we have the commutative diagram in $\textbf{QC}(S)$: 

\begin{equation}
\label{Dia_HessianVSdifferential}
\begin{tikzcd}
i^*\Tt_X \arrow[dr, "Q"'] \arrow[r, "\iota_{df}"] & i^* \_O_X \arrow[d, "d"] \\
& i^*\Ll_X
\end{tikzcd}
\end{equation}

We will abusively write $Q = d \circ \delta : i^*\Tt_X[1] \to i^*\Ll_X$ for the map of degree $1$ corresponding to the composition $d \circ \iota_{df} : i^*\Tt_X \to i^* \Ll_X$ of degree $0$.\\

In general, the natural map $\lambda : S \rightarrow \textbf{Crit}(f)$ is not an equivalence. This is due to the fact that the partial derivatives of $f$ will not in general form a regular sequence and therefore $\textbf{Crit}(f)$ has higher homology. The default to be a regular sequence comes from vector fields that annihilate $df$. Such vector fields are in fact vector fields on $S$ when $f$ is non-degenerate. With that idea in mind, we show that an equivalent description of $\textbf{Crit}(f)$ is given by $T^*[-1]S$ when $Q$ is non-degenerate.\\

\begin{Prop}
\label{Prop_MapExistence}
There exists a natural map $\Phi : T^*[-1]S \rightarrow \textbf{Crit}(f)$ making the following diagram commute: 

$$ \begin{tikzcd}
T^*[-1]S \arrow[d, "\pi_S"] \arrow[r, "\Phi"] & \textbf{Crit}(f) \arrow[d, "\pi"] \\
S \arrow[r,"i"] & X
\end{tikzcd}$$
\end{Prop} 

\begin{proof}
Under our first order splitting assumption (Diagram \eqref{SplittingAssumption}), the natural map $\Tt_S \rightarrow i^*\Tt_X$ admits a retract, and therefore the natural map $i^*T^*X \rightarrow T^*S$ admits a section: $T^*S \dashrightarrow i^* T^*X$. We consider the following diagram:

\begin{equation*}
\begin{tikzcd}
T^*X & \arrow[l] i^*T^*X \arrow[r, shift left] & \arrow[l, dashed, shift left] T^*S  \\
X \arrow[u, "0"] & S \arrow[l, "i"] \arrow[r, equals] \arrow[u, "0"] & S \arrow[u, "0"]
\end{tikzcd}
\end{equation*}

We want to pull-back these zero sections along the maps induced by $df$ represented by the vertical morphisms in the following commutative diagram: 

\begin{equation*}
\begin{tikzcd}
T^*X & \arrow[l] i^*T^*X \arrow[r, shift left] & \arrow[l, dashed, shift left] T^*S  \\
X \arrow[u, "df"] & S \arrow[l, "i"] \arrow[r, equals] \arrow[u, " i^* df = 0"] & S \arrow[u, "0"]
\end{tikzcd}
\end{equation*}  

This induces the following morphisms between the pull-backs: 

$$ \begin{tikzcd}
\textbf{Crit}(f)& \arrow[l] S \times_{i^* T^*X} S \arrow[r, shift left] & T^*[-1]S \arrow[l, dashed
, shift left]
\end{tikzcd}$$

We obtain a map $\Phi : T^*[-1]S \rightarrow \textbf{Crit}(f)$. The maps we obtain come from the universal properties of the pull-backs therefore if we denote $s_0: X \rightarrow T^*X$ the zero section, we have $s_0 \circ \pi \circ \Phi = s_0 \circ i \circ \pi_S$. If we compose by the projection $\pi_X :  T^*X \rightarrow X $, we get $\pi \circ \Phi = i \circ \pi_S$.
\end{proof}

$\Phi$ gives a relationship between the Lagrangian fibration structures on $T^*[-1]S \rightarrow S$ and $\textbf{Crit}(f) \rightarrow X$ which we now analyse. The idea is to show that the difference between these Lagrangian fibrations is in fact controlled by $\widetilde{Q}$ (see Proposition \ref{Prop_NonDegeneracyAndPhiRelatioship} and Remark \ref{RQ_DifferenceLagrangianFibrationAndQ}). 

\begin{Lem}
\label{Lem_NaturalityofLagrangianFibrationNonDegeneracyMorphisms}
$\Phi$ induces a morphism $\Tt_{\faktor{T^*[-1]S}{S}} \rightarrow \Phi^* \Tt_{\faktor{\textbf{Crit}(f)}{X}}$ that fits in the commutative diagram

\begin{equation}
\begin{tikzcd}
\Tt_{\faktor{T^*[-1]S}{S}} \arrow[r] \arrow[d, "\alpha_{\pi_S}"] & \Phi^* \Tt_{\faktor{\textbf{Crit}(f)}{X}} \arrow[d, "\alpha_{\pi}"] \\
\pi_S^* \Ll_S[-1] \arrow[r] & \Phi^* \pi^* \Ll_X[-1] \simeq \pi_S^* i^* \Ll_X[-1]
\end{tikzcd}
\end{equation}
where the bottom horizontal arrow is the pull-back along $\pi_S$ of the section $\Ll_S[-1] \rightarrow i^* \Ll_X [-1]$ in the dual of the split fiber sequence \eqref{SplittingAssumption}. 
\end{Lem}
\begin{proof}
The homotopy pull-back, $\textbf{Crit}(f) = X \times_{T^*X}^h X$ lives over $X$. We get the equivalences:

 \[ \begin{tikzcd}
 \Tt_{\faktor{\textbf{Crit}(f)}{X}} \arrow[r, "\simeq"] & \Tt_{\faktor{X}{X}} \times_{\faktor{\Tt_{T^*X}}{X}}^h \Tt_{\faktor{X}{X}} \arrow[r, "\simeq"] &   \star  \times_{\faktor{\Tt_{T^*X}}{X}}^h \star \arrow[r, "\simeq"] &  \pi^* \Ll_X[-1] 
 \end{tikzcd}\]

Proposition \ref{Prop_CotangentNonDegeneracyFibrationIsNatural} tells us that the canonical fibrations on the cotangent stacks are the canonical ones and therefore behave functorially (using Proposition \ref{Prop_RelativeCotangentComplexFunctorialityForLinearStacks}). This implies that the following commutative square is commutative:
\[ \begin{tikzcd}
\Tt_{\faktor{T^*S}{S}} \arrow[r] \arrow[d, "\beta_S"]& \Tt_{\faktor{T^*X}{X}} \arrow[d, "\beta_X"] \\
\pi_S^* \Ll_S \arrow[r, "\pi_S^* s"] & \pi_S^* i^*\Ll_X
\end{tikzcd} \]
where $s$ is the section in the dual of the split fiber sequence \eqref{SplittingAssumption}. From Proposition \ref{Prop_LagrangianFibrationCanonicalNonDegeneracyMap},
we know that both $\alpha_{\pi_S}$ and $\alpha_{\pi}$ are the morphism induced by the morphisms $\beta_S$ and $\beta_X$ via Diagram \eqref{Eq_NonDegeneracyLagFibDerivedCriticalLocus}. We then obtain the commutative diagram:
\[\begin{tikzcd} \Tt_{\faktor{T^*[-1]S}{S}} \arrow[r, "\simeq"] \arrow[d]& 0 \times_{\Tt_{\faktor{T^*S}{S}}}^h 0 \arrow[r, "0 \times_{\beta_S}^h 0"] \arrow[d] & \pi_S^* \Ll_S[-1] \arrow[d]\\
\Phi^*\Tt_{\faktor{\textbf{Crit}(f)}{X}} \arrow[r, "\simeq"] & \Phi^* \left( 0 \times_{\Tt_{\faktor{T^*X}{X}}}^h 0 \right) \arrow[r, "0 \times_{\beta_X}^h 0"]  & \Phi^* \pi^* \Ll_X[-1]
  \end{tikzcd} \]
  
  where the composition of the horizontal maps are exactly $\alpha_{\pi_S}$ and $\alpha_{\pi}$ thanks to Proposition \ref{Prop_LagrangianFibrationCanonicalNonDegeneracyMap}.
\end{proof}

\begin{Lem}
\label{Lem_diff_Ll_Crit(f)}
We first remark that $\Phi^*\Ll_{\textbf{Crit}(f)}$ can be described, as a sheaf of graded modules (forgetting the differential), by: 

\[ \Phi^*\Ll_{\textbf{Crit}(f)} \simeq \Sym_{\_O_S} \left(\Tt_S[1] \right) \otimes_{\_O_S} \left( i^* \Ll_X \oplus i^* \Tt_X [1] \right) \] 

where $\Ll_X$ is generated by terms of the form $dg$ with $g \in \_O_X$ and $\Tt_X[1]$ is generated by terms of the form $d \xi$ with $\xi \in \Tt_X[1] \subset \_O_{\textbf{Crit}(f)}$. Then, the internal differential on $\Phi^*\Ll_{\textbf{Crit}(f)}$ is characterised by $Q = d \circ \iota_{df}$ via $\delta(d \xi) = Q(\xi)$ and $\delta (dg) = 0$. 
\end{Lem}

\begin{proof}

 The differential on $\Sym_{\_O_S} \left(\Tt_S[1] \right) \otimes_{\_O_S} \left( i^* \Ll_X \oplus i^* \Tt_X [1] \right)$ is $\_O_{T^*[-1]S}$-linear because $\iota_{df}$ is zero on $\Tt_S[1]$. Moreover, for $\xi \in \Tt_X[1] \subset \_O_{\textbf{Crit}(f)}= \Sym_{\_O_X} \Tt_X[1]$, we have  $\delta \circ d (\xi) = d \circ \delta (\xi) = d \circ \iota_{df} (\xi) = Q(\xi)$ (see Diagram \eqref{Dia_HessianVSdifferential}), and for $g\in \_O_X$, $\delta \circ d (g) = d \circ \delta g = 0$. 
\end{proof}

\begin{Lem}
\label{Lem_HessianMapFromLagrangianFibration}
The composition 
\[ \begin{tikzcd} \pi_S^* i^* \Tt_X [-1] \arrow[r] & \Phi^* \Tt_{\faktor{\textbf{Crit}(f)}{X}} \arrow[r, "\alpha_{\pi}"] & \Phi^* \pi^* \Ll_X[-1] \end{tikzcd}\]
is given by $\pi_S^* Q$. Similarly, the composition 

\[ \begin{tikzcd} \pi_S^*  \Tt_S [-1] \arrow[r] & \Tt_{\faktor{T^*[-1]S}{S}} \arrow[r, "\alpha_{\pi_S}"] & \pi_S^* \Ll_S[-1] \end{tikzcd}\]

is $0$ (the restriction of $\pi_S^* Q$ to $S$). 
\end{Lem}
\begin{proof}
 The left morphism is the morphism fitting in the fiber sequence:

\[ \begin{tikzcd} \pi_S^*  i^* \Tt_X [-1] \arrow[r] & \Phi^* \Tt_{\faktor{\textbf{Crit}(f)}{X}} \arrow[r] & \Phi^* \Tt_{\textbf{Crit}(f)} \end{tikzcd}\]

Which gives us:

\[ \begin{tikzcd}
\pi_S^* i^* \Tt_{X} [-1] \arrow[r] \arrow[d, equals] & \Phi^* \Tt_{\faktor{\textbf{Crit}(f)}{X}} \arrow[r] \arrow[d, "\alpha_{\pi}"] & \Phi^* \Tt_{\textbf{Crit}(f)} \arrow[d, "\simeq"] \\
\pi_S^* i^* \Tt_{X} [-1] \arrow[r, dashed] & \Phi^* \pi^* \Ll_{X}[-1] \arrow[r, hookrightarrow]  & \Phi^* \pi^* \Ll_{X}[-1] \oplus \Phi^* \pi^* \Tt_{X}
\end{tikzcd} \]

The second row can be seen as the extension (by $\pi_S^* $) of the fiber sequence: 

\[ \begin{tikzcd}
 i^* \Tt_{X} [-1] \arrow[r, dashed] & i^* \Ll_{X}[-1] \arrow[r, hookrightarrow]  &i^*  \Ll_{X}[-1] \oplus i^* \Tt_{X}
\end{tikzcd} 
\]

Since $X$ and $S$ are smooth, $i^* \Tt_{X} [-1]$ and $i^* \Ll_{X}[-1]$ are both quasi-isomorphic to complexes concentrated in a single degree. This imposes that the dashed arrow is equivalent to the connecting morphism of the induced long exact sequence in cohomology. Therefore, it is equivalent to the map that sends a section $s$ in $i^* \Tt_{X} [-1] $ to its differential, in $i^*  \Ll_{X}[-1] \oplus i^* \Tt_{X}$, which can in turn be seen as an element in   $i^* \Ll_{X}$. More concretely, denote $\tilde{s}$ any lift of $s$ to an element in $i^* \Ll_{X}[-2] \oplus i^* \Tt_{X}   [-1]$. Using Lemma \ref{Lem_diff_Ll_Crit(f)}, its differential is given by 

\[  Q (s) =  Q (\tilde{s}) \in i^*  \Ll_{X}[-1] \subset i^* \Ll_{X}[-1]\oplus i^* \Tt_{X}. \]
 
We then apply $\pi_S^*$ to get the sequence we want. The second part of the statement is proven the same way.
\end{proof}

\begin{Prop}
\label{Prop_NonDegeneracyAndPhiRelatioship}
The map $\Tt_{T^*[-1]S}  \rightarrow  \Phi^* \Tt_{\textbf{Crit}(f)}$ induced by $\Phi$ is an equivalence if and only if $Q$ is non-degenerate.
\end{Prop}

\begin{proof}

First, using the equivalences $\alpha_{\pi}: \Phi^* \Tt_{\faktor{\textbf{Crit}(f)}{X}} \rightarrow \pi_S^* i^* \Ll_X[-1]$ and  $\alpha_{\pi_S} : \Phi^* \Tt_{\faktor{T^*[-1]S}{S}} \rightarrow \pi_S^* \Ll_S[-1]$, we can show that the cofiber of $\Tt_{\faktor{T^*[-1]S}{S}} \rightarrow \Phi^* \Tt_{\faktor{\textbf{Crit}(f)}{X}}$ is equivalent to $\pi_S^* \Ll_{\faktor{S}{X}} [-2]$. Then Lemma \ref{Lem_NaturalityofLagrangianFibrationNonDegeneracyMorphisms} and \ref{Lem_HessianMapFromLagrangianFibration} ensure that the upper half of the following diagram is commutative:

\begin{equation}
\begin{tikzcd}
\pi_S^* \Tt_S[-1] \arrow[d] \arrow[r] & \pi_S^* i^* \Tt_X[-1] \arrow[r] \arrow[d] & \pi_S^* \Tt_{\faktor{S}{X}} \arrow[d, "\widetilde{Q}"] \\
\Tt_{\faktor{T^*[-1]S}{S}} \arrow[r] \arrow[d] & \Phi^* \Tt_{\faktor{\textbf{Crit}(f)}{X}} \arrow[r] \arrow[d] & \pi_S^* \Ll_{\faktor{S}{X}}[-2] \arrow[d]\\
\Tt_{T^*[-1]S} \arrow[r] & \Phi^* \Tt_{\textbf{Crit}(f)} \arrow[r] & \_F
\end{tikzcd}
\end{equation}

This diagram is then commutative and all rows and columns are cofiber sequences and in particular $\_F$ is both the homotopy cofiber of $\Tt_{T^*[-1]S}  \rightarrow  \Phi^* \Tt_{\textbf{Crit}(f)}$ and the homotopy cofiber of $\widetilde{Q}$. In particular, the homotopy cofiber of $\widetilde{Q}$ is zero if and only the homotopy cofiber of $\Tt_{T^*[-1]S}  \rightarrow  \Phi^* \Tt_{\textbf{Crit}(f)}$ is also zero. 
\end{proof}

We now decompose $\alpha_\pi$ into a part along $S$ and a part normal to $S$. This decomposition is by means of split fibered sequences coming from the split fiber sequence \eqref{SplittingAssumption}. \\

\begin{Prop}
When $Q$ is non-degenerate, the maps expressing the non-degeneracy of the Lagrangian fibrations fit in the commutative diagram:
\[ \begin{tikzcd}
\Tt_{\faktor{T^*[-1]S}{S}} \arrow[r] \arrow[d, "\alpha_{\pi_S}"] & \Tt_{\faktor{\textbf{Crit}(f)}{X}} \arrow[r] \arrow[d, "\alpha_{\pi}"] & \Tt_{\faktor{S}{X}} \arrow[d, "\widetilde{Q}"] \\
\pi_S^* \Ll_S [-1] \arrow[r] & \pi_S^* i^* \Ll_X [-1] \arrow[r] & \Ll_{\faktor{S}{X}} [-1]
\end{tikzcd}\]
where the rows are fiber sequences.   
\end{Prop}
\begin{proof}
First, when $Q$ is non-degenerate, the top horizontal sequence is fibered and comes from the following diagram:

\[ \begin{tikzcd}
\Tt_{\faktor{T^*[-1]S}{S} } \arrow[d] \arrow[r] & \Phi^* \Tt_{\faktor{\textbf{Crit}(f)}{X}} \arrow[d] \arrow[r, dashed] & \pi_S^* \Tt_{\faktor{S}{X}} \arrow[d]\\
\Tt_{T^*[-1]S} \arrow[r] \arrow[d] & \Phi^*\Tt_{\textbf{Crit}(f)} \arrow[d] \arrow[r] & 0 \arrow[d] \\
\pi_S^* \Tt_S  \arrow[r]  & \Phi^* i^* \Tt_X \arrow[r]  & \pi_S^* \Tt_{ \faktor{S}{X}}[1]  \\
\end{tikzcd} \]

where all rows and columns are fibered and the cofiber of the second row is $0$ thanks to Proposition \ref{Prop_NonDegeneracyAndPhiRelatioship} since we assumed that $Q$ is non-degenerate. Using Lemma \ref{Lem_NaturalityofLagrangianFibrationNonDegeneracyMorphisms} and Lemma \ref{Lem_HessianMapFromLagrangianFibration}, we obtain the following commutative diagram: 

\begin{equation}\label{Dia_LagrangianFibrationCofiber}
\begin{tikzcd}
\pi_S^* \Tt_S [-1] \arrow[r] \arrow[d] \arrow[dd,  bend right = 70, "0"', near start,crossing over] & \Phi^* i^* \Tt_X [-1] \arrow[r] \arrow[d] \arrow[dd, bend right = 70, "Q"', near start,crossing over] & \pi_S^* \Tt_{\faktor{S}{X}} \arrow[d, equals] \arrow[dd,  bend right = 70, "\widetilde{Q}"', near start,crossing over] \\
\Tt_{\faktor{T^*[-1]S}{S}} \arrow[d, "\alpha_{\pi_S}"] \arrow[r] & \Phi^*\Tt_{\faktor{\textbf{Crit}(f)}{X}} \arrow[d, "\alpha_{\pi}"] \arrow[r] & \pi_S^* \Tt_{\faktor{S}{X}} \arrow[d, dashed] \\
\pi_S^* \Ll_S[-1] \arrow[r] & \Phi^* i^* \Ll_X[-1] \arrow[r] & \pi_S^* \Ll_{\faktor{S}{X}}[-2]
 \end{tikzcd}
  \end{equation}
 
 The only map the dashed arrow can be, in order to make the diagram commutative, is $\widetilde{Q}$. 
\end{proof}

\begin{RQ}
\label{RQ_DifferenceLagrangianFibrationAndQ}
If we do not assume $Q$ non-degenerate, the cofiber $\_F$ of the map $\Tt_{T^*[-1]S} \rightarrow \Phi^* \Tt_{\textbf{Crit}(f)}$ will be non zero. We will denote by $\_G$ the fiber of the natural map $\_F \rightarrow \Tt_{\faktor{S}{X}}$. Then we can rewrite Diagram \eqref{Dia_LagrangianFibrationCofiber} as

\[
\begin{tikzcd}
\pi_S^* \Tt_S [-1] \arrow[r] \arrow[d] \arrow[dd,  bend right = 70, "0"', near start,crossing over] & \Phi^* i^* \Tt_X [-1] \arrow[r] \arrow[d] \arrow[dd, bend right = 70, "Q"', near start,crossing over] & \pi_S^* \Tt_{\faktor{S}{X}} \arrow[d] \arrow[dd,  bend right = 70, "\widetilde{Q}"', near start,crossing over] \\
\Tt_{\faktor{T^*[-1]S}{S}} \arrow[d, "\alpha_{\pi_S}"] \arrow[r] & \Phi^*\Tt_{\faktor{\textbf{Crit}(f)}{X}} \arrow[d, "\alpha_{\pi}"] \arrow[r] & \_G \arrow[d, "\alpha_N"] \\
\pi_S^* \Ll_S[-1] \arrow[r] & \Phi^* i^* \Ll_X[-1] \arrow[r] & \pi_S^* \Ll_{\faktor{S}{X}}[-2]
 \end{tikzcd}
  \]

 The map $\alpha_N : \_G \rightarrow \pi_S^*\Ll_{\faktor{S}{X}}[-2] $ represent the "difference" between the maps $\alpha_\pi$ and $\alpha_{\pi_S}$ from the Lagrangian fibrations. $\alpha_N$ is still related to $\widetilde{Q}$ in the sense that the following diagram is commutative: 
 
 \[\begin{tikzcd}
  \Tt_{\faktor{S}{X}} \arrow[d] \arrow[dr, "\widetilde{Q}"]  & \\
 \_G \arrow[r, "\alpha_N"] & \Ll_{\faktor{S}{X}} [-2]
 \end{tikzcd}\]
 
 Therefore the restriction of $\alpha_N$ to $\Tt_{\faktor{S}{X}}$ is again $\widetilde{Q}$. 
\end{RQ}

\begin{RQ}
As a non-example if we take $f : \Aa^1 \rightarrow \Aa^1$ sending $X$ to $\frac{X^3}{3}$, the basic assumptions that made this section work are failing. The strict critical locus $S$ is not smooth since it is a fat point, and the sequence \eqref{SplittingAssumption} does not split.  
\end{RQ}   

\subsection{Derived Zero Locus of Shifted 1-Forms}
\label{Sec_DerivedIntersectionClosedForms}
Let $X$ be a derived Artin stack and $\alpha \in \_A^{1} \left( X, n \right)$ be a 1-form. If $\textbf{Key}(\alpha )$ is non-empty, Proposition \ref{Prop_SpaceofKeyequivalenttoSpaceofIsotropicStructures} and Remark \ref{RQ_Closed1FormAreNonDegenerateIsotropicStructure} ensure that the map $\alpha : X \rightarrow T^*[n]X$ is a Lagrangian morphism. Using Theorem \ref{Th_DerivedIntersectionLagrangianFibration} the derived intersection $Z(\alpha)$ of $\alpha$ with the zero section gives us a Lagrangian fibration $Z(\alpha) \rightarrow X$. This example is a generalisation of the derived critical locus we described in \ref{Sec_DerivedCriticalLocus}.

\subsection{$G$-Equivariant Twisted Cotangent Bundles}

 For $X$ a smooth scheme, a twisted cotangent stack is a twist of the ordinary cotangent stack by a closed $1$-form of degree $1$ on $X$, $\alpha \in H^1(X, \Omega_X^1)$. Such a closed form has an underlying 1-form of degree 1 that corresponds to a morphism $\alpha : X \rightarrow T^*[1]X$. The \defi{twisted cotangent bundle} associated to $\alpha$ is defined to be the following pull-back:

$$ \begin{tikzcd}
   T_\alpha^* X \arrow[r] \arrow[d] & X \arrow[d, "\alpha"] \\
   X \arrow[r, "0"] & T^*[1]X
 \end{tikzcd}
 $$ 

We refer to \cite{Ha16} for more informations on the relation between this definition and the usual definition of twisted cotangent bundles. This is a particular case of the situation in Section \ref{Sec_DerivedIntersectionClosedForms} and as such, $T_\alpha^* X$ is $0$-shifted symplectic and the map $T_\alpha^* X \rightarrow X$ has a Lagrangian fibration structure. \\

Now take $G$ an algebraic group acting on the algebraic variety $X$. Consider a character $\chi : G \rightarrow \Gg_m$. We have the logarithmic form on $\Gg_m$ given by a map $\Gg_m \rightarrow \_A^{1,cl}(-,0)$ which sends $z$ to $z^{-1} dz$. We get a closed 1-form on $G$ described by the composition:

$$ G \rightarrow \Gg_m \rightarrow \_A^{1,cl}(-,0) $$

This is also a group morphism for the additive structure on $\_A^{1,cl}(-,0)$. We can therefore pass to classifying spaces and obtain a $1$-shifted closed $1$-form on $\textbf{B}G$:

$$ \alpha_\chi : \textbf{B}G \rightarrow \textbf{B}\_A^{1,cl}(-,0) = \_A^{1,cl}(-,1)$$

We can consider the pull-back of $\alpha_\chi$ along the $G$-equivariant moment map:

$$ \begin{tikzcd}
\left[ \faktor{T^*X}{G} \right] \times_{ \left[ \faktor{\G_g^*}{G} \right] } \textbf{B}G \arrow[r] \arrow[d] & \textbf{B}G \arrow[d, "\alpha_\chi"] \\
\left[ \faktor{T^*X}{G} \right]  \arrow[r, "\mu"] & \left[ \faktor{\G_g^*}{G} \right]  \simeq T^*[1]\textbf{B}G
\end{tikzcd}$$ 

It turns out that the moment map $\mu$ is Lagrangian (see \cite{Cal1}), which implies (with Proposition \ref{Prop_LagrangianIntersectionareShiftedSymplectic}) that this fiber product is $0$-shifted symplectic. It turns out that we have an equivalence of shifted symplectic derived Artin stacks: 

$$\left[ \faktor{T^*X}{G} \right] \times_{ \left[ \faktor{\G_g^*}{G} \right] }\textbf{B}G \simeq T_{\widehat{\alpha}}^* \left[ \faktor{X}{G} \right] $$ 

Where $\widehat{\alpha}$ denotes the pull-back of $\alpha_\chi$ to a $1$-form of degree $1$ on $\left[ \faktor{X}{G} \right]$. Therefore, according to Theorem \ref{Th_DerivedIntersectionLagrangianFibration}, the natural projection
 
 $$ \begin{tikzcd}
 T_{\widehat{\alpha}}^* \left[ \faktor{X}{G} \right] \arrow[r] & \left[ \faktor{X}{G} \right]
 \end{tikzcd}$$
 is a Lagrangian fibration.\\

To show the equivalence above, we use the following composition of the following Lagrangian correspondences (see \ref{Def_LagrangianCorrespondence}):

\begin{itemize}
	\item The Lagrangian structure on the section $ \left[ \faktor{X}{G} \right] \rightarrow T^*[1] \left[ \faktor{X}{G} \right]$: 
	
	\[ \begin{tikzcd}
		& \left[ \faktor{X}{G} \right] \arrow[dl] \arrow[dr, "0"]& \\
		\star& &  T^*[1] \left[ \faktor{X}{G} \right] 
	\end{tikzcd}\]

\item Using Example 2.3 in \cite{Cal1}, and the fact that $\left[ \faktor{X \times \G_g^*}{G} \right] \simeq \left[\faktor{X}{G}\right] \times_{\left[ \faktor{\star}{G} \right]} \left[ \faktor{\G_g^*}{G} \right] $ we obtain the Lagrangian correspondence: 

\[ \begin{tikzcd}
	& \left[ \faktor{X \times \G_g^*}{G} \right] \simeq \left[\faktor{X}{G}\right] \times_{\left[ \faktor{\star}{G} \right]} \left[ \faktor{\G_g^*}{G} \right] \arrow[dl] \arrow[dr]& \\
	T^* [1]\left[ \faktor{X}{G} \right] & & \left[ \faktor{\G_g^*}{G} \right] \simeq T^*[1]\left[ \faktor{\star}{G} \right] 
\end{tikzcd}\]

\item The Lagrangian obtain from the closed $1$-form of degree $1$, $\alpha_\chi$:  	

\[ \begin{tikzcd}
	& \textbf{B}G \arrow[dl, "\alpha_\chi"] \arrow[dr]& \\
 \left[ \faktor{\G_g^*}{G} \right] 	& & \star
\end{tikzcd}\]

\end{itemize}

We then compose these Lagrangian correspondences: 

$$ 
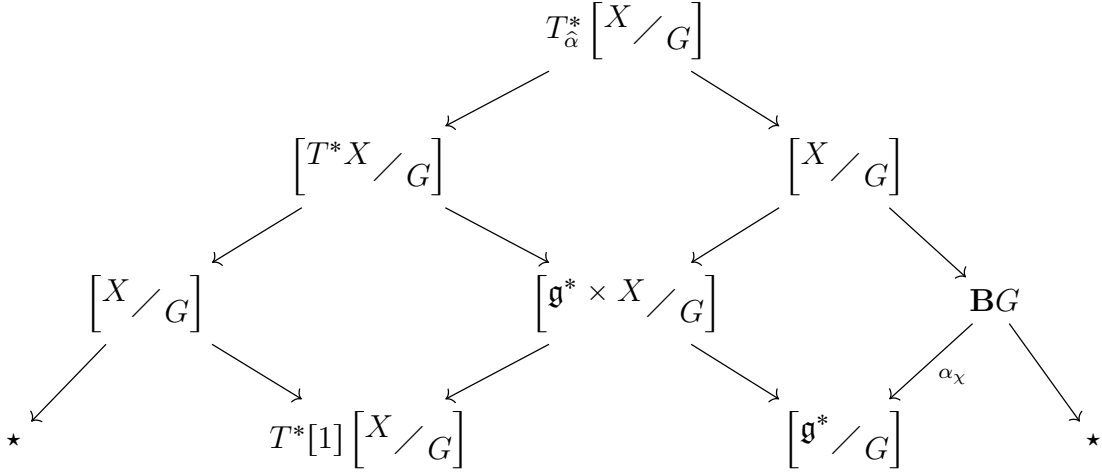
\begin{tikzcd}[column sep=small]
  & & & T_{\widehat{\alpha}}^* \left[ \faktor{X}{G} \right] \arrow[dl] \arrow[dr]& & &\\
  & & \left[ \faktor{T^*X}{G} \right] \arrow[dl] \arrow[dr] & & \left[ \faktor{X}{G} \right] \arrow[dl] \arrow[dr] & & \\
  &\left[ \faktor{X}{G} \right] \arrow[dl] \arrow[dr]& & \left[ \faktor{\G_g^* \times X}{G} \right] \arrow[dl] \arrow[dr] & &\textbf{B}G \arrow[dl, "\alpha_\chi"] \arrow[dr] & \\
  \star & & T^*[1] \left[ \faktor{X}{G} \right]& & \left[ \faktor{\G_g^*}{G} \right]& & \star 
\end{tikzcd}
$$

The only thing we need to show is that this is a diagram of Lagrangian correspondences and therefore we need to show that all squares in this diagrams are pull-backs. The right most square is clearly a pull-back and we can recognise the pull-back square defining $T_{\widehat{\alpha}}^* \left[ \faktor{X}{G} \right]$.\\

 We are left to prove that we have a natural equivalence:
 \[ \left[ \faktor{X}{G} \right] \times_{ T^*[1]\left[ \faktor{X}{G} \right]} \left[ \faktor{ \G_g^* \times X}{G} \right] \simeq \left[ \faktor{T^*X}{G} \right] \]

We can commute taking the quotient by (compatible) $G$-action and taking the fiber products so we have that:

\[ \left[ \faktor{X}{G} \right] \times_{ T^*[1]\left[ \faktor{X}{G} \right]} \left[ \faktor{ \G_g^* \times X}{G} \right]  \simeq   \left[ \faktor{X}{G} \right] \times_{ T^*[1]\left[ \faktor{X}{G} \right]}\left[ \faktor{X}{G} \right] \times_{\left[ \faktor{\star}{G} \right]} \left[ \faktor{ \G_g^*}{G} \right]\]

We now use the fact that the self intersection of the $0$ section in $T^*[1]\left[\faktor{X}{G} \right]$ is $T^*\left[\faktor{X}{G} \right]$. This implies that:

\[ \left[ \faktor{X}{G} \right] \times_{ T^*[1]\left[ \faktor{X}{G} \right]} \left[ \faktor{ \G_g^* \times X}{G} \right]  \simeq   T^*\left[ \faktor{X}{G} \right] \times_{\left[ \faktor{\star}{G} \right]} \left[ \faktor{ \G_g^*}{G} \right]\]

 We can now usethe fact fact the following square is a pull-back (Example 2.2.1 in \cite{Saf16}):
\[
\begin{tikzcd}
	T^* \left[ \faktor{X}{G} \right] \arrow[d] \arrow[r] & \textbf{B}G \arrow[d, "0"]\\
	\left[ \faktor{T^*X}{G} \right] \arrow[r, "\mu"] & \left[ \faktor{\G_g^* }{G} \right] 
\end{tikzcd}\]

We use that to decompose $	T^* \left[ \faktor{X}{G} \right]$ in a fiber product and we obtain:

\[\left[ \faktor{X}{G} \right] \times_{ T^*[1]\left[ \faktor{X}{G} \right]} \left[ \faktor{ \G_g^* \times X}{G} \right]  \simeq \left[ \faktor{T^*X}{G} \right] \times_{\left[ \faktor{\G_g^* }{G} \right]} \left[ \faktor{\star}{G}  \right] \times_{\left[ \faktor{\star}{G} \right]} \left[ \faktor{\G_g^*}{G} \right] \simeq \left[ \faktor{T^*X}{G} \right] \]


\nocite{*}
\bibliographystyle{plain}
\bibliography{LagrangianFibration}

\begin{thebibliography}{10}

\bibitem{BF}
Kai Behrend and Barbara Fantechi.
\newblock Symmetric obstruction theories and {H}ilbert schemes of points on
  threefolds.
\newblock {\em Algebra Number Theory}, 2(3):313--345, 2008.

\bibitem{Darboux}
Christopher Brav, Vittoria Bussi, and Dominic Joyce.
\newblock A {D}arboux theorem for derived schemes with shifted symplectic
  structure.
\newblock {\em J. Amer. Math. Soc.}, 32(2):399--443, 2019.

\bibitem{Bu}
Vittoria Bussi.
\newblock {\em Derived symplectic structures in generalized
  {D}onaldson--{T}homas theory and categorification}.
\newblock 2014.
\newblock Thesis (D.Phil.)--University of Oxford (United Kingdom).

\bibitem{Cal3}
Damien Calaque.
\newblock Three lectures on derived symplectic geometry and topological field
  theories.
\newblock {\em Indag. Math. (N.S.)}, 25(5):926--947, 2014.

\bibitem{Cal4}
Damien Calaque.
\newblock Lagrangian structures on mapping stacks and semi-classical {TFT}s.
\newblock In {\em Stacks and categories in geometry, topology, and algebra},
  volume 643 of {\em Contemp. Math.}, pages 1--23. Amer. Math. Soc.,
  Providence, RI, 2015.

\bibitem{Cal2}
Damien Calaque.
\newblock Shifted cotangent stacks are shifted symplectic.
\newblock {\em Ann. Fac. Sci. Toulouse Math. (6)}, 28(1):67--90, 2019.

\bibitem{Cal1}
Damien Calaque.
\newblock {\em Derived Stacks in Symplectic Geometry}, volume~2, page
  155–201.
\newblock Cambridge University Press, 2021.

\bibitem{CPTVV}
Damien Calaque, Tony Pantev, Bertrand To\"{e}n, Michel Vaqui\'{e}, and Gabriele
  Vezzosi.
\newblock Shifted {P}oisson structures and deformation quantization.
\newblock {\em J. Topol.}, 10(2):483--584, 2017.

\bibitem{Ha16}
M{\'a}rton {Hablicsek}.
\newblock {Derived intersections over the Hochschild cochain complex}.
\newblock {\em arXiv e-prints}, page arXiv:1608.06965, August 2016.

\bibitem{Joyce}
Dominic Joyce.
\newblock A classical model for derived critical loci.
\newblock {\em J. Differential Geom.}, 101(2):289--367, 2015.

\bibitem{LuDAG}
Jacob Lurie.
\newblock {\em Derived algebraic geometry}.
\newblock Massachusetts Institute of Technology, 2004.
\newblock Thesis (Ph.D.)--Massachusetts Institute of Technology.

\bibitem{LuDAGV}
Jacob {Lurie}.
\newblock {Derived Algebraic Geometry V: Structured Spaces}.
\newblock {\em arXiv e-prints}, page arXiv:0905.0459, May 2009.

\bibitem{PTVV}
Tony Pantev, Bertrand To\"{e}n, Michel Vaqui\'{e}, and Gabriele Vezzosi.
\newblock Shifted symplectic structures.
\newblock {\em Publ. Math. Inst. Hautes \'{E}tudes Sci.}, 117:271--328, 2013.

\bibitem{Saf16}
Pavel Safronov.
\newblock Quasi-hamiltonian reduction via classical chern–simons theory.
\newblock {\em Advances in Mathematics}, 287:733--773, 2016.

\bibitem{STV}
Timo Sch\"{u}rg, Bertrand To\"{e}n, and Gabriele Vezzosi.
\newblock Derived algebraic geometry, determinants of perfect complexes, and
  applications to obstruction theories for maps and complexes.
\newblock {\em J. Reine Angew. Math.}, 702:1--40, 2015.

\bibitem{SP}
The {Stacks project authors}.
\newblock The stacks project.
\newblock \url{https://stacks.math.columbia.edu}, 2020.

\bibitem{To2}
Bertrand To\"{e}n.
\newblock Champs affines.
\newblock {\em Selecta Math. (N.S.)}, 12(1):39--135, 2006.

\bibitem{To1}
Bertrand To\"{e}n.
\newblock Derived algebraic geometry.
\newblock {\em EMS Surv. Math. Sci.}, 1(2):153--240, 2014.

\bibitem{HAGI}
Bertrand To\"{e}n and Gabriele Vezzosi.
\newblock Homotopical algebraic geometry. {I}. {T}opos theory.
\newblock {\em Adv. Math.}, 193(2):257--372, 2005.

\bibitem{HAGII}
Bertrand To\"{e}n and Gabriele Vezzosi.
\newblock Homotopical algebraic geometry. {II}. {G}eometric stacks and
  applications.
\newblock {\em Mem. Amer. Math. Soc.}, 193(902):x+224, 2008.

\bibitem{Ve1}
Gabriele Vezzosi.
\newblock Basic structures on derived critical loci.
\newblock {\em Differential Geometry and its Applications}, 71:101635, 2020.

\bibitem{Xu03}
Ping Xu.
\newblock Momentum maps and morita equivalence.
\newblock {\em J. Differential Geom.}, 67(2):289--333, 06 2004.

\end{thebibliography}

 \end{document}